\documentclass[11pt]{article}
\usepackage{amsfonts, amstext, amsmath, amsthm, amscd, amssymb, hyperref, float, longtable, fullpage}
\usepackage{epsfig, graphics, psfrag}
\usepackage{graphicx}
\usepackage{color}
\usepackage[font=small, labelfont=bf]{caption}
\usepackage{float}

\newtheorem{definition}{Definition}[section] 
\let\olddefinition\definition
\renewcommand{\definition}{\olddefinition\normalfont}
\newtheorem{remark}{Remark}[section]
\let\oldremark\remark
\renewcommand{\remark}{\oldremark\normalfont}
\newtheorem{proposition}{Proposition}[section]
\newtheorem{theorem}{Theorem}[section]

\newtheorem{corollary}{Corollary}[section]

\newtheorem{notation}{Notation}[section]

\newtheorem{example}{Example}[section]
\let\oldexample\example
\renewcommand{\example}{\oldexample\normalfont}

\let\emptyset\varnothing

\begin{document}

\title{Sigma-Adequate Link Diagrams and the Tutte Polynomial}
\author{Adam Giambrone\\ University of Connecticut\\ adam.giambrone@uconn.edu}
\date{}
\maketitle

\begin{abstract}
In this paper, we characterize the sigma-adequacy of a link diagram in two ways: in terms of a certain edge subset of its Tait graph and in terms of a certain product of Tutte polynomials. Furthermore, we show that the symmetrized Tutte polynomial of the Tait graph of a link diagram can be written as a sum of these products of Tutte polynomials, where the sum is over the sigma-adequate states of the given link diagram. Using this state sum, we show that the number of sigma-adequate states of a link diagram is bounded above by the number of spanning trees in its associated Tait graph. By combining results, we give a method to find all of the sigma-adequate states of a link diagram. Finally, we give necessary and sufficient conditions for a link diagram to be sigma-adequate and sigma-homogeneous (also called homogeneously adequate) with respect to a given state. 
\end{abstract}

\section{Introduction}
Since the early years of knot theory, a correspondence between link diagrams and edge-signed planar graphs (called \textit{Tait graphs}) has been known to exist (\cite{Przytycki}). As a means to study both links and graphs, polynomial invariants have been used. An important polynomial invariant of graphs, introduced in 1947, is the two-variable \textit{Tutte polynomial} (\cite{Tutte}) and an important polynomial invariant of links, introduced in 1984, is the one-variable \textit{Jones polynomial} (\cite{Jones1}, \cite{Jones2}). As shown by Kauffman in \cite{Bracket}, the Jones polynomial satisfies a recursive skein relation that parallels the contraction-deletion relation of the Tutte polynomial. Deepening the connection between the Jones and Tutte polynomials, Thistlethwaite expressed the Jones polynomial of an alternating link in terms of the Tutte polynomial of its associated Tait graph (\cite{Thistlethwaite1}) and, soon after, Kauffman extended this result to all links by defining a Tutte polynomial for edge-signed graphs (\cite{Kauffman}). 

Kauffman also generalized the Jones polynomial to a two-variable Laurent polynomial that is now called the \textit{Kauffman polynomial} (\cite{KauffPoly}). Using the (unnormalized) Kauffman polynomial of a link diagram, Thistlethwaite (\cite{Thistlethwaite2}) extracted two \textit{boundary term polynomials} $\phi_{D}^{+}(t)$ and $\phi_{D}^{-}(t)$ (called \textit{critical line polynomials} in \cite{Stoimenow2}) and expressed each such polynomial as a product of two Tutte polynomials, where the Tutte polynomials come from edge-contractions and edge-deletions of the associated Tait graph. From this, Thistlethwaite shows that a link diagram is \textit{A-adequate} (resp. \textit{B-adequate}) if any only if the boundary term polynomial $\phi_{D}^{+}(t)$ (resp. $\phi_{D}^{-}(t)$) is nonvanishing. 

The family of \textit{semi-adequate} links, that is, links that have either A- or B-adequate diagrams, was introduced by Lickorish and Thistlethwaite (\cite{LT}, \cite{Lickorish}) and is a very large family of links. For example, the family of semi-adequate links has been shown to contain all alternating links, all adequate links, all Montesinos links, all positive and negative closed braids, all closed 3-braids, and all planar cables of the link families just mentioned (\cite{LT}, \cite{Stoimenow}). To give a sense of how frequently semi-adequate links occur, we have that all knots with at most 10 crossings are semi-adequate and all but two 11-crossing prime knots are semi-adequate (\cite{Stoimenow2}). As one would expect, the family of semi-adequate links has been widely studied (\cite{TailAdeq}, \cite{HeadTail}, \cite{New}, \cite{EffieChristine}, \cite{Christine}, \cite{LT}, \cite{Lickorish}, \cite{Stoimenow}, \cite{Stoimenow2}, \cite{Thistlethwaite2}). While many knots and links are known to be semi-adequate, it has been shown (\cite{LT}, \cite{Stoimenow2}, \cite{Thistlethwaite2}) that there exist links that are neither A- nor B-adequate (such links are sometimes called \textit{inadequate}). 


In this paper, we extend some of the results of Thistlethwaite from \cite{Thistlethwaite2} to the larger family of \textit{$\sigma$-adequate} link diagrams, where $\sigma$ denotes a fixed but arbitrary state of a link diagram. To begin, let $D$ denote a checkerboard-colored link diagram with associated Tait graph $G$. Given that the edges of $G$ are labeled with $+$ or $-$ and given that the crossings of $D$ are either A-resolved or B-resolved according to the state $\sigma$, we use $E_{\sigma}$ to denote the union of the $+$ edges of $G$ whose corresponding crossings are A-resolved and the $-$ edges of $G$ whose corresponding crossings are B-resolved. Given the edge subset $E_{\sigma} \subseteq E(G)$, let $G|E_{\sigma}$ denote $G$ restricted to the edge set $E_{\sigma}$ and let $\overline{E_{\sigma}}=E(G)-E_{\sigma}$ denote the complement of the edge set $E_{\sigma}$. The first main result of this paper is given below. 

\begin{theorem} \label{introthm}
Let $D$ be a connected, reduced, checkerboard-colored link diagram with at least one crossing and with associated Tait graph $G$. Then $D$ is $\sigma$-adequate with respect to a state $\sigma$ if and only if there exists a partition $E(G)=E_{\sigma} \sqcup \overline{E_{\sigma}}$ of the edges of $G$ such that the following conditions hold.  
\begin{enumerate}
	\item[(1)] Every edge of $E_{\sigma}$ is contained in a cycle of $G|E_{\sigma}$. 
  \item[(2)] No edge of $\overline{E_{\sigma}}$ has both endpoints on a connected component of $G|E_{\sigma}$. 
\end{enumerate}
\end{theorem}

To generalize the boundary term polynomials of Thistlethwaite (\cite{Thistlethwaite2}) we define, for each state $\sigma$ of a link diagram, a polynomial $\phi_{D}^{\sigma}(t)$ that is a product of two Tutte polynomials and that reduces to a boundary term polynomial when the state being considered is the all-A state or the all-B state. Using the polynomial $\phi_{D}^{\sigma}(t)$, we extend Thistlethwaite's nonvanishing result to the family of $\sigma$-adequate link diagrams. 

\begin{theorem} \label{mainthm1}
Let $D$ be a connected checkerboard-colored link diagram with at least one crossing and with associated Tait graph $G$. Then $D$ is $\sigma$-adequate with respect to a state $\sigma$ if and only if $\phi_{D}^{\sigma}(t) \neq 0$. 
\end{theorem}

Given the Tutte polynomial $\chi_{G}(x,y)$ of the Tait graph $G$ associated to a checkerboard-colored link diagram $D$, call $\chi_{G}(t,t)$ the \textit{symmetrized Tutte polynomial} of $G$. The third main result of this paper, stated below, gives a $\sigma$-adequate state sum expansion for $\chi_{G}(t,t)$. 

\begin{theorem} \label{mainthm2}
Let $D$ be a connected checkerboard-colored link diagram with at least one crossing and with associated Tait graph $G$. Then the symmetrized Tutte polynomial of $G$ can be expanded as $$\chi_{G}(t,t)=\sum_{\sigma\text{-adequate\ states}\ \sigma}\phi_{D}^{\sigma}(t).$$ 
\end{theorem}

Using this state sum, we are able to bound the number of $\sigma$-adequate states of a link diagram. 

\begin{theorem} \label{maincor1}
Let $D$ be a connected, reduced, checkerboard-colored link diagram with at least one crossing and with associated Tait graph $G$. Then the number of $\sigma$-adequate states of $D$ is bounded below by two and above by the number of spanning trees in $G$.
\end{theorem}

By combining the results of Theorem~\ref{introthm}, Theorem~\ref{mainthm1}, and Theorem~\ref{mainthm2}, we present a method for finding all of the $\sigma$-adequate states of a connected, reduced, checkerboard-colored link diagram. We highlight this method by using it to find all 20 $\sigma$-adequate states of the standard diagram (\cite{KnotInfo}) of the nonalternating inadequate knot $11n_{95}$.  

Next, we apply the notion of partial duality, as defined by Chmutov in \cite{Chmutov}, to the ribbon graph associated to the black checkerboard state of a checkerboard-colored link diagram in order to give necessary and sufficient conditions for a link diagram to be $\sigma$-adequate and $\sigma$-homogeneous (also called \textit{homogeneously adequate}) with respect to a given state. The family of links with homogeneously adequate diagrams has been studied by a number of authors (\cite{Homogeneous}, \cite{Guts}, \cite{Survey}, \cite{Ozawa}).     

\begin{theorem} \label{maincor2}
Let $D$ be a connected, reduced, checkerboard-colored link diagram with at least one crossing and with associated Tait graph $G$. Then $D$ is homogeneously adequate with respect to a state $\sigma$ if and only if there exists a partition $E(G)=E_{\sigma} \sqcup \overline{E_{\sigma}}$ of the edges of $G$ such that the following conditions hold.
\begin{enumerate}
	\item[(1)] Every edge of $E_{\sigma}$ is contained in a cycle of $G|E_{\sigma}$. 
  \item[(2)] No edge of $\overline{E_{\sigma}}$ has both endpoints on a connected component of $G|E_{\sigma}$. 
	\item[(3)] Each connected component of $G|E_{\sigma}$ either consists entirely of $+$ edges or consists entirely of $-$ edges. 
	\item[(4)] The edges of $\overline{E_{\sigma}}$ inside a fundamental cycle of $G|E_{\sigma}$ either consist entirely of $+$ edges or consist entirely of $-$ edges. 
 	\item[(5)] The edges of $\overline{E_{\sigma}}$ in the unbounded region outside all of the fundamental cycles of $G|E_{\sigma}$ either consist entirely of $+$ edges or consist entirely of $-$ edges. 
\end{enumerate}
\end{theorem} 

Finally, we conclude the paper by modifying the method mentioned above to give a method for finding all of the homogeneously adequate states of a given link diagram. As an example, we show that none of the 20 $\sigma$-adequate states of the standard diagram of the knot $11n_{95}$ are $\sigma$-homogeneous. Thus, we show that this knot diagram has no homogeneously adequate states. 

This paper is organized as follows. In Section~\ref{background}, we provide background on the Tait graph associated to a checkerboard-colored link diagram, the Tutte polynomial of a graph, the $\sigma$-adequacy (as well as A-, B-, and semi-adequacy) of a link diagram, and the unnormalized Kauffman polynomial and boundary term polynomials associated to a link diagram. In Section~\ref{extending}, we generalize the boundary term polynomials to the polynomials $\phi_{D}^{\sigma}(t)$ associated to any state $\sigma$ of a link diagram; we provide a characterization of the $\sigma$-adequacy of a link diagram that involves edge-contractions and edge-deletions of the Tait graph; we prove Theorem~\ref{introthm}, Theorem~\ref{mainthm1}, Theorem~\ref{mainthm2}, and Theorem~\ref{maincor1}; and we provide a method for finding all of the $\sigma$-adequate states of a link diagram. In Section~\ref{duality}, we define the state ribbon graphs of a link diagram, define the operation of partial duality for these ribbon graphs, and build connections between the state ribbon graph perspective and the Tait graph perspective for link diagrams. In Section~\ref{homogmethod}, we introduce the families of $\sigma$-homogeneous and homogeneously adequate link diagrams, we prove Theorem~\ref{maincor2}, and we present a method for finding all of the homogeneously adequate states of a link diagram.


\section{Background} \label{background}

In this section, we present the background from graph theory and knot theory that will be used in the rest of this paper.

\subsection{The Tait Graph Associated to a Link Diagram}

Let $D \subseteq \mathbb{R}^{2}$ denote a diagram of a link $K$. By replacing the crossings of $D$ with 4-valent vertices, we get the \textit{projection graph} $\Gamma_{D}$ associated to $D$. We call $D$ \textit{connected} if $\Gamma_{D}$ is connected. From $D$ we can form its \textit{mirror reflection}, denoted $\overline{D}$, by changing the crossing types of all of the crossings of $D$. 

A link diagram $D$ is called \textit{checkerboard-colored} if each complementary region of $\Gamma_{D}$, which we call a \textit{region} of $D$, is shaded either black or white so that, in a small neighborhood of each crossing, two opposite regions are shaded black and two opposite regions are shaded white. See the left side of Figure~\ref{Taitgraph} for an example of a checkerboard-colored link diagram. It is a well-known result that every link diagram can be checkerboard-colored. 

\begin{figure} 
\centering
\def\svgwidth{1.25in}
	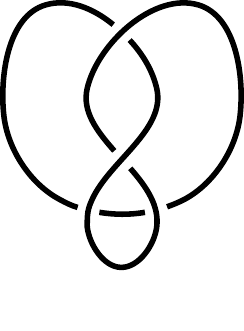 
	\hspace{.5in}
\def\svgwidth{1in}
	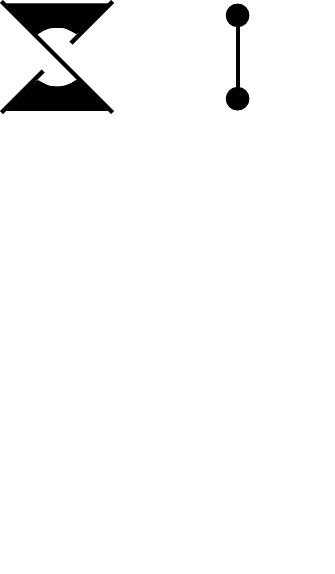 	
	\hspace{.5in}
\def\svgwidth{1.25in}
	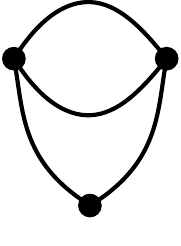 
	\caption{An example of the bijection between a checkerboard-colored link diagram $D$ and its corresponding Tait graph $G$.} 
	\label{Taitgraph}
\end{figure}

\begin{definition} \label{Tait}
To a checkerboard-colored link diagram $D$, we construct the associated \textit{Tait graph} $G$ as follows. First, we associate a vertex to each black region of $D$. Next, we join a pair of distinct vertices by an edge if the two corresponding regions share a crossing of $D$ and we join a vertex to itself if the corresponding region meets a crossing of $D$ twice. Finally, we label the edges with either $+$ or $-$ according to the convention displayed in the center of Figure~\ref{Taitgraph}. What results is an edge-signed planar graph $G$ called the \textit{Tait graph} associated to $D$. See Figure~\ref{Taitgraph} for an example of the construction of a Tait graph. Note that the white regions of $D$ correspond to the faces (complementary regions) of $G$. Also note that changing the crossing type of a single crossing of $D$ corresponds to changing the sign of the corresponding edge of $G$ and forming the mirror reflection $\overline{D}$ of $D$ corresponds to changing all of the edge signs of $G$.   
\end{definition}

\begin{remark} \label{Taitrmk}
As was first observed by Tait in the late 1800s (for the case of alternating knots), there is a bijection between the family of checkerboard-colored link diagrams and the family of edge-signed planar graphs. The construction defined above gives one direction of this bijection and a medial graph construction can be used to give the reverse direction. Refer to \cite{Kauffman} for more details. As a special case of this bijection, alternating checkerboard-colored link diagrams correspond to edge-signed planar graphs where the edges are either all $+$ edges or all $-$ edges. See the top half of Figure~\ref{checkerstate}.  
\end{remark}

\begin{figure}
\centering
\def\svgwidth{5in}
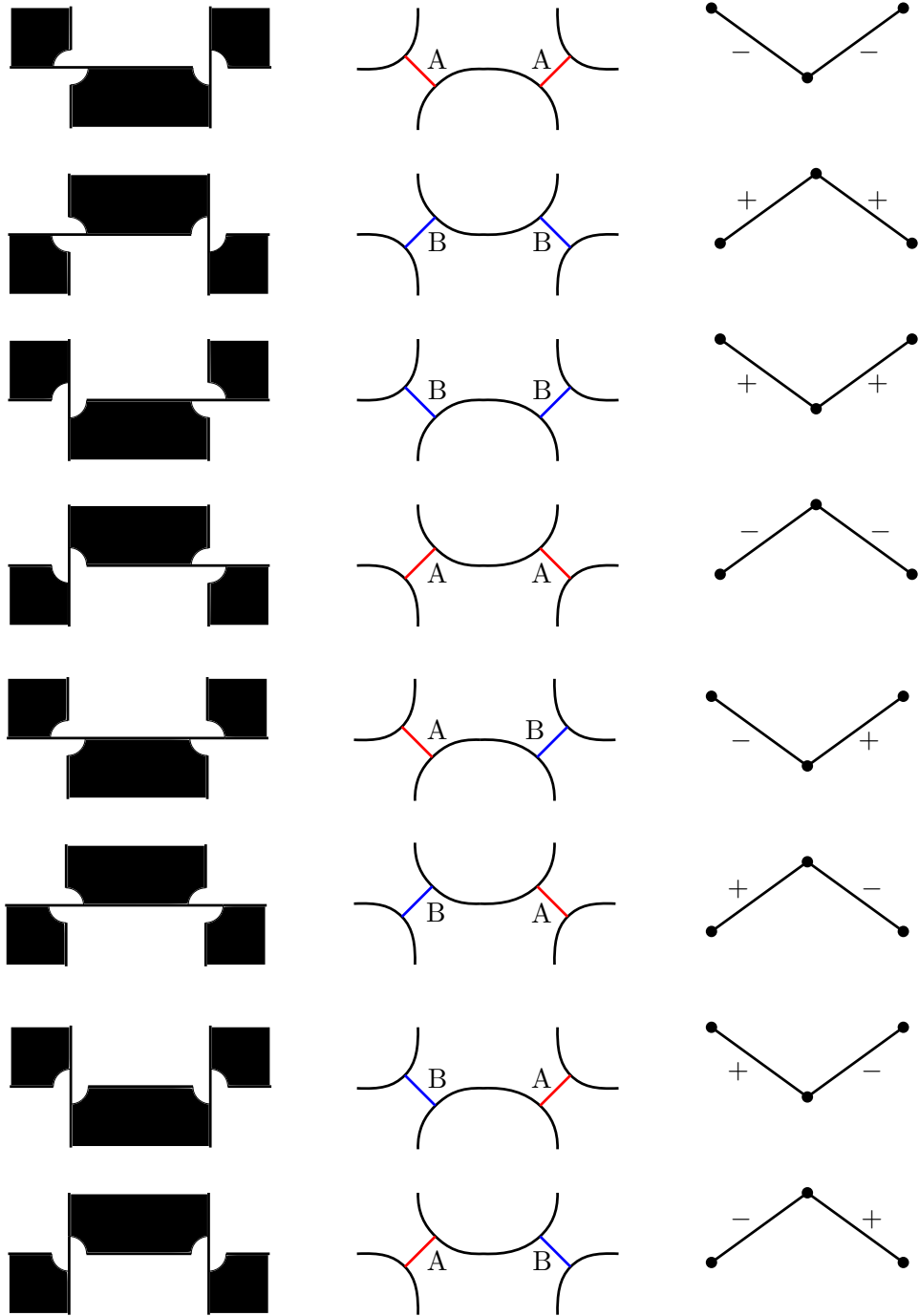
\caption{The local structure for two adjacent crossings of a checkerboard-colored link diagram $D$ (left), the corresponding local behavior of the black checkerboard state graph $H_{\sigma_{bl}}$ (center), and the corresponding local behavior of the associated Tait graph $G$ (right).}
\label{checkerstate}
\end{figure}




Throughout this paper, we will often work with \textit{reduced} link diagrams.  

\begin{definition}
A link diagram $D$ is called \textit{reduced} if it contains no \textit{nugatory crossings}, that is, crossings that meet a region of $D$ more than once. 
\end{definition}


We now introduce some terminology from graph theory. 

\begin{definition} \label{loopdef}
Let $G$ be a graph. An edge of $G$ is called a \textit{loop} if it joins a vertex to itself. An edge $e=\left\{v,w\right\}$ of $G$ is called a \textit{multiedge} if there are multiple \textit{parallel edges} of $G$ that join vertex $v$ to vertex $w$. An edge of $G$ is called a \textit{bridge} if deleting this edge from $G$ increases the number of connected components. A vertex of $G$ is called a \textit{cut vertex} if deleting this vertex and its incident edges from $G$ increases the number of connected components. A graph $G$ is called \textit{nonseparable} if it is connected and contains no cut vertices. Every graph $G$ can be decomposed into \textit{blocks}, where a \textit{block} is either an isolated vertex of $G$, a bridge of $G$ with its incident vertices, a loop of $G$ with its incident vertex, or a maximally connected subgraph of $G$ that contains no cut vertices and is neither a bridge nor a loop of $G$. The blocks and cut vertices of a connected graph $G$ give rise to a \textit{block decomposition} of $G$ as $G=G_{1} \vee G_{2} \vee \cdots \vee G_{k}$, where the $G_{i}$ are the blocks of $G$ and where the $G_{i}$ meet at shared cut vertices of $G$. 
\end{definition}


The remark below, which relates properties of a checkerboard-colored link diagram to properties of its associated Tait graph, will be used multiple times throughout this paper. 

\begin{remark} \label{nugatoryrmk}
Let $D$ be a checkerboard-colored link diagram with associated Tait graph $G$. By Definition~\ref{Tait} and Remark~\ref{Taitrmk}, the following statements can be shown to be true. 
\begin{enumerate}
	\item[(1)] The connected components of $D$ correspond to the connected components of $G$. 
	\item[(2)] $D$ is connected if and only if $G$ is connected. 
	\item[(3)] Nugatory crossings of $D$ correspond to either bridges or loops of $G$. 
	\item[(4)] $D$ is reduced (contains no nugatory crossings) if and only if $G$ contains no bridges and no loops. 
\end{enumerate}
\end{remark}

\subsection{Graph Operations and the Tutte Polynomial} \label{Tuttestuff}

The graph operations of \textit{deletion} and \textit{contraction} with respect to both a single edge and a collection of edges will play an important role in the remainder of this paper. 

\begin{definition} \label{deletedef}
Let $G$ be a graph. Then, for $e$ an edge of $G$, we use $G-e$ to denote the graph obtained from $G$ by \textit{deleting the edge} $e$ and $G/e$ to denote the graph obtained from $G$ by \textit{contracting the edge} $e$. Our convention will be that contracting a loop is the same as deleting the loop. Let $H \subseteq E(G)$ denote a subset of the edges of $G$. Then we use $G|{H}$ to denote $G$ \textit{restricted to the edges of} $H$, that is, $G|H$ is the (spanning) subgraph of $G$ whose vertices are the same as those of $G$ and whose edges are the edges of $H$. Alternatively, $G|H$ is the graph obtained from $G$ by \textit{deleting the edges in the complement} $\overline{H}=E(G)-H$ of the edge set $H$. Finally, let $G/H$ denote the graph obtained from $G$ by \textit{contracting the edges of} $H$. 
\end{definition}

The Tutte polynomial of a graph is defined below. 

\begin{definition} \label{Tuttepoly}
The \textit{Tutte polynomial} of a graph $G$, denoted $\chi_{G}(x, y)$, is a two-variable polynomial invariant defined recursively as follows. 
\begin{enumerate}
	\item[(1)] If $G=G_{1} \sqcup G_{2} \sqcup \cdots \sqcup G_{k}$ is a disjoint union of $k$ graphs, then $$\chi_{G}(x,y)=\chi_{G_{1}}(x,y)\cdot\chi_{G_{2}}(x,y)\cdots\chi_{G_{k}}(x,y).$$  
	\item[(2)] If $G=G_{1} \vee G_{2} \vee \cdots \vee G_{k}$ is a block decomposition of a connected graph $G$ into $k$ blocks, then $$\chi_{G}(x,y)=\chi_{G_{1}}(x,y)\cdot\chi_{G_{2}}(x,y)\cdots\chi_{G_{k}}(x,y).$$ 
  \item[(3)] If $G$ is nonseparable, then $$\chi_{G}(x, y)=\left\{
\begin{array}{ll}
1 & \hspace{.25in} \mathrm{if}\ G\ \mathrm{contains\ no\ edges}\\
x \cdot \chi_{G/e}(x, y) & \hspace{.25in} \mathrm{if}\ e\ \mathrm{is\ a\ bridge\ of}\ G\\
y \cdot \chi_{G-e}(x, y) & \hspace{.25in} \mathrm{if}\ e\ \mathrm{is\ a\ loop\ of}\ G\\
\chi_{G/e}(x, y)+\chi_{G-e}(x, y) & \hspace{.25in} \mathrm{if}\ e\ \mathrm{is\ neither\ a\ bridge\ nor\ a\ loop\ of}\ G\\
\end{array}
\right.
$$
\end{enumerate}
\end{definition}

By repeatedly applying Definition~\ref{Tuttepoly}, we are able to compute the Tutte polynomial of the cycle graph $C_{n}$ with $n$ edges. 

\begin{proposition} \label{Tuttecycleprop}
$\chi_{C_{n}}(x,y)=x^{n-1}+x^{n-2}+\cdots+x^{2}+x+y$. \qed
\end{proposition}

Recall that a graph is called \textit{planar} if it can be drawn in the plane $\mathbb{R}^{2}$ in such a way that no two edges cross each other. For a planar graph, the faces (complementary regions) are well-defined. We now define the planar dual of a planar graph. 

\begin{definition} \label{planardualdef}
The \textit{planar dual} of a planar graph $G$, denoted $G^{*}$, is the planar graph where the vertices of $G^{*}$ correspond to the faces of $G$, the faces of $G^{*}$ correspond to the vertices of $G$, and the edges of $G^{*}$ correspond to the edges of $G$ so that an edge and its planar dual intersect transversely when we overlay $G$ with its planar dual $G^{*}$. 
\end{definition}

\begin{remark} \label{planardualrmk}
Given a planar graph $G$ with planar dual $G^{*}$, it can be shown that the bridges (resp. loops) of $G^{*}$ correspond to the loops (resp. bridges) of $G$. It can also be shown that contraction and deletion reverse roles under the operation of planar duality. Thus, for $e$ an edge of $G$ and $e^{*}$ the corresponding dual edge of $G^{*}$, we have that $(G/e)^{*} \cong G^{*}-e^{*}$ and $(G-e)^{*} \cong G^{*}/e^{*}$. Furthermore, in the case that $G$ is the Tait graph associated to a checkerboard-colored link diagram $D$, it can be seen that changing the checkerboard coloring of $D$, which reverses the black and white regions of $D$, corresponds to changing from $G$ to its planar dual $G^{*}$.    
\end{remark}

The following proposition, which can be proved by applying Remark~\ref{planardualrmk} to Definition~\ref{Tuttepoly}, shows how the Tutte polynomial behaves under the operation of planar duality. 

\begin{proposition} \label{planardualprop}
If $G$ is a planar graph with planar dual $G^{*}$, then $\chi_{G^{*}}(x,y)=\chi_{G}(y,x)$. \qed
\end{proposition}

\subsection{Sigma-Adequacy for Link Diagrams} \label{sigmaadeqsec}

To define the family of $\sigma$-adequate link diagrams, we first need a precise definition of a state $\sigma$ of a link diagram. 

\begin{definition} \label{statedef}
Let $D$ be a link diagram. We may either \textit{A-resolve} or \textit{B-resolve} a crossing of $D$ according to Figure~\ref{res}. A \textit{state} $\sigma$ of $D$ is a choice of A- or B-resolution at each crossing. Given a state $\sigma$ of $D$, the corresponding \textit{dual state}, denoted $\sigma^{*}$, is the state where the choices of A- and B-resolution from the state $\sigma$ are all reversed. The \textit{all-A (resp. all-B) state} of $D$, denoted $\sigma_{A}$ (resp. $\sigma_{B}$), is the state that results from choosing the A-resolution (resp. B-resolution) at each crossing of $D$. Note that $\sigma_{A}$ and $\sigma_{B}$ are dual states.   
\end{definition}

Given a state of a checkerboard-colored link diagram we define, in the remark below, four types of edges of the associated Tait graph. 

\begin{remark} \label{edgelabel}
Recall, from Definition~\ref{Tait}, that the edges of the Tait graph $G$ associated to a checkerboard-colored link digram $D$ are labeled with either $+$ or $-$. Fix a state $\sigma$ of $D$. Since either the A- or B-resolution has been chosen at each crossing of $D$ and since the crossings of $D$ correspond to the edges of $G$, then we can label the edges of $G$ with either $A$ or $B$. Therefore, given a state of a checkerboard-colored link digram, the associated Tait graph has edges that can be labeled $(+,A)$, $(+,B)$, $(-,A)$, or $(-,B)$. 
\end{remark}

\begin{figure}
\centering
\def\svgwidth{3in}
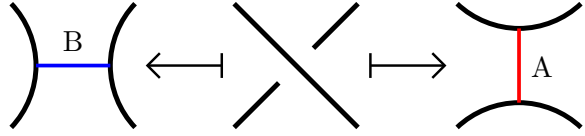
\caption{The A- and B-resolutions at a crossing of a link diagram $D$. The A-segment is colored red and the B-segment is colored blue. For grayscale versions of this paper, the A-segments will appear light gray and the B-segments will appear dark gray.}
\label{res}
\end{figure}
 
\begin{figure}
\centering
\def\svgwidth{3in}
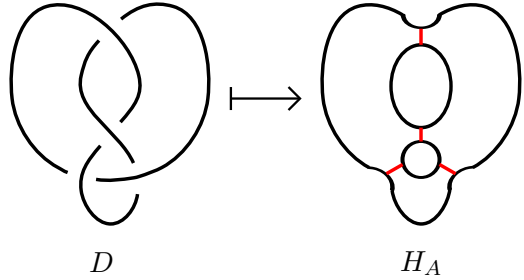
\caption{The standard diagram $D$ of the figure-8 knot (left) and its all-A state graph $H_{A}$ (right).}
\label{stategraph}
\end{figure}

\begin{definition} \label{trivalent}
By resolving the crossings of a link diagram $D$ according to a state $\sigma$, we form the trivalent \textit{state graph} $H_{\sigma}$ that consists of a disjoint collection of \textit{state circles} and a disjoint collection of \textit{state segments} (either \textit{A-segments} or \textit{B-segments}) that are used to record the locations and types of the crossing resolutions. Let $H_{A}$ (resp. $H_{B}$) denote the state graph that arises from the all-A (resp. all-B) state. For an example of a link diagram and its all-A state graph, see Figure~\ref{stategraph}. Note that we may sometimes suppress the labels of A and B on the state segments of $H_{\sigma}$.
\end{definition}

We now define what it means for a link diagram to be $\sigma$-, A-, B-, and semi-adequate. 

\begin{definition} \label{adeq}
A link diagram $D$ is called \textit{$\sigma$-adequate with respect to a state $\sigma$} if its state graph $H_{\sigma}$ contains no state segments that connect a state circle to itself. We call such a state a \textit{$\sigma$-adequate state} of $D$. As a special case, $D$ is called \textit{A-adequate} (resp. \textit{B-adequate}) if $H_{A}$ (resp. $H_{B}$) contains no state state segments that connect a state circle to itself and $D$ is called \textit{semi-adequate} if it is either A- or B-adequate. For an example of an A-adequate link diagram, see Figure~\ref{stategraph}. Note that $D$ is A-adequate if and only if its mirror reflection $\overline{D}$ is B-adequate.  
\end{definition}

Two important states of a link diagram are its \textit{checkerboard states}. 

\begin{definition} \label{checkerboardstatesdef}
Given a checkerboard-colored link diagram $D$, the \textit{black (resp. white) checkerboard state} of $D$, denoted $\sigma_{bl}$ (resp. $\sigma_{wh}$), is the state of $D$ where the crossings of $D$ are resolved so that the state circles of $H_{\sigma_{bl}}$ (resp.  $H_{\sigma_{wh}}$) are, up to minor perturbations, the boundaries of the black (resp. white) regions of $D$. Note that $\sigma_{bl}$ and $\sigma_{wh}$ are dual states. 
\end{definition}

As will be shown below, the black and white checkerboard states of a reduced link diagram are always $\sigma$-adequate states. 

\begin{proposition} \label{checkerrmk}
Every reduced link diagram $D$ has at least two $\sigma$-adequate states, namely the black and white checkerboard states coming from a checkerboard coloring of $D$. Furthermore, $D$ is alternating if and only if the two checkerboard states, $\sigma_{bl}$ and $\sigma_{wh}$, of $D$ coincide with the all-A and all-B states, $\sigma_{A}$ and $\sigma_{B}$, of $D$.
\end{proposition}

\begin{proof} Since $D$ is reduced, then, by Remark~\ref{nugatoryrmk} and Remark~\ref{planardualrmk}, neither its Tait graph $G$ nor its planar dual $G^{*}$ contain loops. By Definition~\ref{Tait}, Definition~\ref{planardualdef}, and Definition~\ref{checkerboardstatesdef}, the vertices of $G$ (resp. $G^{*}$) correspond to the black (resp. white) regions of $D$, which correspond to the state circles of the black (resp. white) checkerboard state graph $H_{\sigma_{bl}}$ (resp. $H_{\sigma_{wh}}$). Moreover, the edges of $G$ (resp. $G^{*}$) correspond to the crossings of $D$, which correspond to the state segments of $H_{\sigma_{bl}}$ (resp. $H_{\sigma_{wh}}$). Therefore, loops of $G$ (resp. $G^{*}$) correspond to state segments of $H_{\sigma_{bl}}$ (resp. $H_{\sigma_{wh}}$) that connect a state circle to itself. Thus, by Definition~\ref{adeq}, we have the first desired result. The second desired result follows from considering Figure~\ref{checkerstate}.  
\end{proof}

\subsection{The Kauffman and Boundary Term Polynomials}

Let $K$ be an oriented link and let $D$ be a diagram of $K$ with writhe $w(D)$. The \textit{Kauffman polynomial}, $F_{K}(a, z)$, of $K$ is an invariant defined by $F_{K}(a, z)=a^{-w(D)}\cdot \Lambda_{D}(a, z)$, where $\Lambda_{D}=\Lambda_{D}(a, z)$ is the \textit{unnormalized Kauffman polynomial} defined recursively as follows. 
\begin{enumerate}
\item[(1)] $\Lambda_{D}=1$ if $D$ is the trivial diagram of the unknot. 
\item[(2)] $\Lambda_{D_{+1}}=a\cdot\Lambda_{D}$ and $\Lambda_{D_{-1}}=a^{-1}\cdot\Lambda_{D}$, where $D_{+1}$ (resp. $D_{-1}$) denotes the result of applying Reidermeister move R1 to $D$ to increase the number of crossings of $D$ by one and increase (resp. decrease) $w(D)$ by one. 
\item[(3)] $\Lambda_{D}$ is invariant under Reidermeister moves R2 and R3. 
\item[(4)] $\Lambda_{D}$ satisfies the skein relation $\Lambda_{D_{+}}+\Lambda_{D_{-}}=z\cdot\left(\Lambda_{D_{\infty}}+\Lambda_{D_{0}}\right)$, where $D_{+}$, $D_{-}$, $D_{\infty}$, and $D_{0}$ denote the diagrams defined locally in Figure~\ref{skein}. 
\end{enumerate}


\begin{figure} 
\centering
\def\svgwidth{4in}
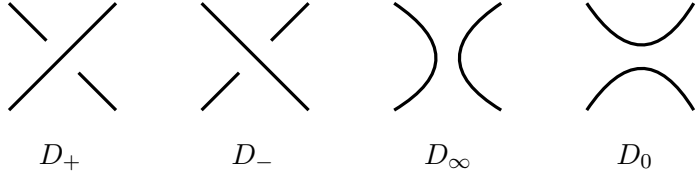
\caption{The diagrams $D_{+}$, $D_{-}$, $D_{\infty}$, and $D_{0}$ involved in the skein relation for the unnormalized Kauffman polynomial. These four diagrams agree outside of the local regions depicted above.}
\label{skein}
\end{figure}

Let $D$ be a link diagram with $n$ crossings. Expanding $\Lambda_{D}(a, z)$ as $\displaystyle \Lambda_{D}(a, z)=\sum_{r, s}\ u_{rs}\cdot a^{r}z^{s}$, Thistlethwaite shows that $u_{rs}\neq 0$ only if $\left|r\right|+s\leq n$ (Theorem 4 of \cite{Thistlethwaite3}). 

\begin{definition} \label{KP}
Call a coefficient $u_{rs}$ of $\Lambda_{D}(a, z)$ \textit{outermost} if either $r+s=n$ or $-r+s=n$. From the outermost coefficients of $\Lambda_{D}(a, z)$, we can construct the \textit{boundary term polynomials} $$\phi_{D}^{+}(t)=\sum_{i\geq 0}\ u_{i, n-i}\cdot t^{i}\ \ \ \ \ \text{and}\ \ \ \ \ \phi_{D}^{-}(t)=\sum_{i\geq 0}\ u_{-i, n-i}\cdot t^{i}.$$ 
\end{definition}



In \cite{Thistlethwaite2}, Thistlethwaite expresses each of the two boundary term polynomials $\phi_{D}^{+}(t)$ and $\phi_{D}^{-}(t)$ as a product of two Tutte polynomials and shows that $\phi_{D}^{+}(t)$ (resp. $\phi_{D}^{-}(t)$) is nonvanishing if and only if $D$ is A-adequate (resp. B-adequate).  

\begin{proposition}{(Theorem 1 and Theorem 3 of \cite{Thistlethwaite2})} \label{nonvanishing}
Let $D$ be a connected checkerboard-colored link diagram with at least one crossing and with associated Tait graph $G$. Let $E_{+}$ denote the set of positive edges of $G$ and let $E_{-}$ denote the set of negative edges of $G$. Then 
\begin{enumerate}
\item[(1)] $\phi_{D}^{+}(t)=\chi_{G|E_{+}}(0, t)\cdot\chi_{G/E_{+}}(t, 0)$ and $\phi_{D}^{-}(t)=\chi_{G|E_{-}}(0, t)\cdot\chi_{G/E_{-}}(t, 0)$.
\item[(2)] $\displaystyle \phi_{D}^{+}(t) \neq 0$ if and only if $D$ is A-adequate and $\displaystyle \phi_{D}^{-}(t) \neq 0$ if and only if $D$ is B-adequate. \qed 
\end{enumerate}
\end{proposition}



\section{Extending to the Sigma-Adequate Case} \label{extending}

In this section, we define the polynomials $\phi_{D}^{\sigma}(t)$ that generalize Thistlethwaite's boundary term polynomials, we prove Theorem~\ref{introthm}, Theorem~\ref{mainthm1}, Theorem~\ref{mainthm2}, and Theorem~\ref{maincor1}, and we provide a method for finding all of the $\sigma$-adequate states of a link diagram. 

\subsection{Generalizing the Boundary Term Polynomials} 

To begin, we define the polynomials that extend Thistlethwaite's boundary term polynomials to every state of a link diagram. 

\begin{definition} \label{phi}
Let $D$ be a checkerboard-colored link diagram with associated Tait graph $G$ and let $\sigma$ be a state of $D$. Using Remark~\ref{edgelabel}, let $E_{+}^{A}$ denote the $(+,A)$ edges of $G$, let $E_{-}^{B}$ denote the $(-,B)$ edges of $G$, and let $E_{\sigma} \subseteq E(G)$ denote the edge set $E_{\sigma}=E_{+}^{A} \cup E_{-}^{B}$. Given the edge set $E_{\sigma}$, define the polynomial $\phi_{D}^{\sigma}(t)$ to be $$\phi_{D}^{\sigma}(t)=\chi_{G|{E_{\sigma}}}(0,t) \cdot \chi_{G/{E_{\sigma}}}(t,0).$$ Furthermore, if we let $E_{+}^{B}$ denote the $(+,B)$ edges of $G$ and let $E_{-}^{A}$ denote the $(-,A)$ edges of $G$, then the complement $\overline{E_{\sigma}}=E(G)-E_{\sigma}$ of the edge set $E_{\sigma}$ can be defined as $\overline{E_{\sigma}}=E_{+}^{B} \cup E_{-}^{A}$. Note that $\overline{E_{\sigma}}=E_{\sigma^{*}}$, where $\sigma^{*}$ is the state dual to the state $\sigma$. 
\end{definition}

If we apply the definition of $\phi_{D}^{\sigma}(t)$ above to the all-A and all-B states of $D$, then we get the boundary term polynomials introduced by Thistlethwaite in \cite{Thistlethwaite2}. 

\begin{remark} \label{phisignrmk}
By Definition~\ref{phi}, if $\sigma_{A}$ is the all-A state of a link diagram $D$, then $$\phi_{D}^{\sigma_{A}}(t)=\chi_{G|{E_{\sigma_{A}}}}(0,t) \cdot \chi_{G/{E_{\sigma_{A}}}}(t,0)=\chi_{G|{E_{+}}}(0,t) \cdot \chi_{G/{E_{+}}}(t,0)=\phi_{D}^{+}(t).$$ Similarly, if $\sigma_{B}$ is the all-B state of $D$, then $\phi_{D}^{\sigma_{B}}(t)=\phi_{D}^{-}(t)$. 
\end{remark}


In the following proposition, we apply Definition~\ref{phi} to find the edge subsets $E_{\sigma}$ corresponding the two checkerboard states of a link diagram. 
 
\begin{proposition} \label{signlemma}
Let $D$ be a connected, reduced, checkerboard-colored link diagram with at least one crossing and with associated Tait graph $G$. Then, given the black checkerboard state $\sigma_{bl}$ of $D$, every edge of $G$ is either a $(+,B)$ edge or a $(-,A)$ edge of $G$ and, given the white checkerboard state $\sigma_{wh}$ of $D$, every edge of $G$ is either a $(+,A)$ edge or a $(-,B)$ edge of $G$. Thus, we have that $E_{\sigma_{bl}} = \emptyset$ and $E_{\sigma_{wh}} = E(G)$. Consequently, we get that $\phi_{D}^{\sigma_{bl}}(t)=\chi_{G}(t,0)$ and $\phi_{D}^{\sigma_{wh}}(t)=\chi_{G}(0,t)$.    
\end{proposition}

\begin{proof}
Given the Tait graph $G$ and the black checkerboard state $\sigma_{bl}$, Remark~\ref{edgelabel} allows us to classify the edges of $G$ as $(+,A)$, $(+,B)$, $(-,A)$, or $(-,B)$ edges. By considering the local structure of the link diagram $D$, the black checkerboard state graph $H_{\sigma_{bl}}$, and the associated Tait graph $G$, we get that the edges of $G$ are either $(+,B)$ or $(-,A)$ edges. See Figure~\ref{checkerstate}. Therefore, by Definition~\ref{phi}, we have that $E_{\sigma_{bl}} = \emptyset$. Since $G|{E_{\sigma_{bl}}}$ contains no edges and since $G/{E_{\sigma_{bl}}}=G$, then, by Definition~\ref{phi} and Definition~\ref{Tuttepoly}, we get that $\phi_{D}^{\sigma_{bl}}(t)=\chi_{G|{E_{\sigma_{bl}}}}(0,t) \cdot \chi_{G/{E_{\sigma_{bl}}}}(t,0)=1 \cdot \chi_{G}(t,0) = \chi_{G}(t,0)$. The proof for the white checkerboard state $\sigma_{wh}$ is similar.  
\end{proof}

We now establish two properties of the polynomial $\phi_{D}^{\sigma}(t)$ that will be used later in this paper. The proposition below generalizes Corollary~1.1 of \cite{Thistlethwaite2}. It should be noted that the three additional properties that appear in Corollary~1.1 of \cite{Thistlethwaite2} can also be generalized to apply to the polynomial $\phi_{D}^{\sigma}(t)$. 

\begin{proposition} \label{phiprop}
Let $D$ be a connected checkerboard-colored link diagram with $n \geq 1$ crossings and with associated Tait graph $G$. Then the polynomial $\phi_{D}^{\sigma}(t)$ corresponding to a state $\sigma$ of $D$ has the following properties. 
\begin{enumerate}
	\item[(1)] The coefficients of $\phi_{D}^{\sigma}(t)$ are nonnegative. 
	\item[(2)] $\phi_{D}^{\sigma}(t) \neq 0$ if and only if $G|E_{\sigma}$ contains no bridges and $G/E_{\sigma}$ contains no loops.
\end{enumerate}
\end{proposition}

\begin{proof}[Proof of Property (1)]
By Definition~\ref{Tuttepoly}, it can be seen that the coefficients of any Tutte polynomial are nonnegative. Hence, by Definition~\ref{phi}, the coefficients of $\phi_{D}^{\sigma}(t)$ are nonnegative.  
\end{proof}

\begin{proof}[Proof of Property (2)] 
($\Rightarrow$) Proceed by contraposition. If $G|{E_{\sigma}}$ contains a bridge, then, by Definition~\ref{Tuttepoly}, we have that $\chi_{G|{E_{\sigma}}}(0,t)=0$. If $G/{E_{\sigma}}$ contains a loop, then, by Definition~\ref{Tuttepoly}, we have that $\chi_{G/{E_{\sigma}}}(t,0)=0$. In either case, by Definition~\ref{phi}, we get that $\phi_{D}^{\sigma}(t)=0$.

($\Leftarrow$) Proceed by contradiction. Assume that $G|{E_{\sigma}}$ contains no bridges, that $G/{E_{\sigma}}$ contains no loops, and that $\phi_{D}^{\sigma}(t)=0$. Therefore, by Definition~\ref{phi}, either $\chi_{G|{E_{\sigma}}}(0,t)=0$ or $\chi_{G/{E_{\sigma}}}(t,0)=0$. Suppose $\chi_{G|{E_{\sigma}}}(0,t)=0$. Then, by the contrapositive of Proposition~2(ii) of \cite{Thistlethwaite1}, we can conclude that either $G|{E_{\sigma}}$ contains a bridge or $G|E_{\sigma}$ contains a loop. Since $G|E_{\sigma}$ contains no bridges by assumption, then $G|{E_{\sigma}}$ contains a loop. If $E_{L}$ denotes the set of loops of $G|{E_{\sigma}}$, then, by Definition~\ref{Tuttepoly}, we have that $\chi_{G|{E_{\sigma}}}(0,t)=t^{|E_{L}|} \cdot \chi_{\left(G|E_{\sigma}\right)-E_{L}}(0,t)$. But then $\left(G|E_{\sigma}\right)-E_{L}$ contains neither bridges nor loops. Therefore, by Proposition~2(ii) of \cite{Thistlethwaite1}, we have that $\chi_{\left(G|E_{\sigma}\right)-E_{L}}(0,t)~\neq~0$, which implies that $\chi_{G|{E_{\sigma}}}(0,t) \neq 0$, a contradiction. The proof of the case when $\chi_{G/E_{\sigma}}(t,0)=0$ is similar. 
\end{proof}

\subsection{A Tait Graph Perspective on Sigma-Adequacy}

Our goal is now to show that the $\sigma$-adequacy of a connected checkerboard-colored link diagram $D$ can be detected from properties of the edge-restricted Tait graph $G|E_{\sigma}$ and the edge-contracted Tait graph $G/E_{\sigma}$. 

\begin{notation}
For $\sigma$ a state of a link diagram $D$, let $\left|s_{\sigma}(D)\right|$ denote the number of state circles in the state graph $H_{\sigma}$. 
\end{notation}

In the proposition below, we provide an alternate definition of $\sigma$-adequacy for a link diagram. 

\begin{proposition} \label{adeqaltprop}
A link diagram $D$ is $\sigma$-adequate with respect to a state $\sigma$ if and only if we have $\left|s_{\sigma}(D_{c}')\right|<\left|s_{\sigma}(D)\right|$, where $D_{c}'$ is the link diagram obtained from $D$ by changing the crossing type of a single crossing $c$ of $D$. 
\end{proposition}

\begin{proof}
By Definition~\ref{adeq}, $D$ is $\sigma$-adequate with respect to a state $\sigma$ if and only if its state graph $H_{\sigma}$ contains no state segments that connect a state circle to itself. This is equivalent to the condition that every state segment $s$ of $H_{\sigma}$ connects two distinct state circles of $H_{\sigma}$, call them $C_{1}$ and $C_{2}$. Since the state segment $s$ corresponds to a crossing, call it $c$, of $D$, then changing the crossing type of $c$ (but not changing the state $\sigma$) switches the local structure of the crossing resolution at $c$ so that the new state segment, call it $s_{c}'$, now joins a state circle, call it $C_{12}$, to itself. Therefore, the condition that every state segment $s$ of $H_{\sigma}$ connects two distinct state circles $C_{1}$ and $C_{2}$ of $H_{\sigma}$ is equivalent to the condition that $\left|s_{\sigma}(D_{c}')\right|=\left|s_{\sigma}(D)\right|-1$, where $D_{c}'$ is the link diagram obtained from $D$ by changing the crossing type of a single crossing $c$. Since changing the type of a single crossing of $D$ can either increase $\left|s_{\sigma}(D)\right|$ by one, leave $\left|s_{\sigma}(D)\right|$ unchanged, or decrease $\left|s_{\sigma}(D)\right|$ by one, then the equation $\left|s_{\sigma}(D_{c}')\right|=\left|s_{\sigma}(D)\right|-1$ is equivalent to the inequality $\left|s_{\sigma}(D_{c}')\right|<\left|s_{\sigma}(D)\right|$, which gives the desired result.  
\end{proof}



We now give a state circle count for a link diagram and use it to provide a characterization of $\sigma$-adequacy for a link diagram that uses the edge-restricted and edge-contracted Tait graphs. 

\begin{proposition} \label{circleprop}
Let $D$ be a connected checkerboard-colored link diagram with $n \geq 0$ crossings and with associated Tait graph $G$. Furthermore, let $e(G)$ denote the number of edges of $G$, let $v(G)$ denote the number of vertices of $G$, and let $k(G)$ denote the number of connected components of $G$. Then, for $\sigma$ a state of $D$, we have the following.
\begin{enumerate}
	\item[(1)] $\left|s_{\sigma}(D)\right|=e(G|E_{\sigma})-v(G|E_{\sigma})+2k(G|E_{\sigma})$. 
	\item[(2)] $D$ is $\sigma$-adequate with respect to the state $\sigma$ if and only if $G|E_{\sigma}$ contains no bridges and $G/E_{\sigma}$ contains no loops. 
\end{enumerate}
\end{proposition}

\begin{proof}[Proof of Conclusion (1)]
Proceed by induction on $n$ (which is both the number of crossings of $D$ and, by Definition~\ref{Tait}, the number of edges of $G$). In the case when $n=0$, we have that $D$ is the trivial diagram of the unknot and, therefore, $\left|s_{\sigma}(D)\right|=1$. Furthermore, since $E_{\sigma}=\emptyset$, then $G|E_{\sigma}=G$ is a single isolated vertex and $$e(G|E_{\sigma})-v(G|E_{\sigma})+2k(G|E_{\sigma})=0-1+2(1)=1.$$ Now assume that $\left|s_{\sigma}(D)\right|=e(G|E_{\sigma})-v(G|E_{\sigma})+2k(G|E_{\sigma})$ for $D$ a connected checkerboard-colored link diagram with $n = j \geq 0$ crossings. Let $n = j+1 \geq 1$ and choose an arbitrary crossing $c$ of $D$, which corresponds to an edge $e$ of $G$. 

\bigskip

\noindent \underline{Case 1}: Suppose $c$ is a non-nugatory crossing of $D$. By Remark~\ref{nugatoryrmk}, this implies that $e$ is neither a bridge nor a loop of $G$. Since there are two crossing types and two state resolutions, then there are four subcases to consider. By rotation, we do not need to consider two separate cases for the two checkerboard colorings of $D$.  

\bigskip

\begin{figure} 
\centering
\def\svgwidth{4in}
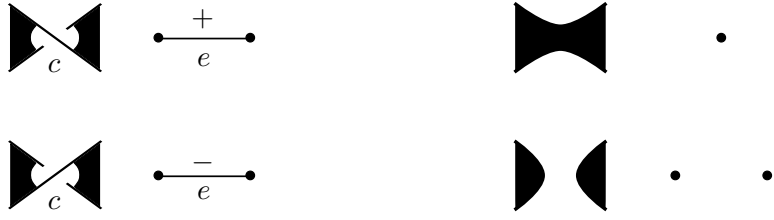
\caption{Schematic for non-nugatory crossings of checkerboard-colored link diagrams and the associated edges in the Tait graph (left) and the two resolutions of these non-nugatory crossings and the associated local behavior of their Tait graphs (right).}
\label{nug4}
\end{figure}

\noindent \underline{Subcase 1}: Suppose a local neighborhood of $c$ has the checkerboard coloring represented by the top left of Figure~\ref{nug4} and suppose $c$ is A-resolved by the state $\sigma$. Then, by Remark~\ref{edgelabel} and Definition~\ref{phi}, we have that $e$ is a $(+,A)$ edge of $G$ and, therefore, $e \in E_{\sigma}$. Let $D'$ be the link diagram formed by A-resolving the crossing $c$ and let $\sigma'$ be the state $\sigma$ restricted to the crossings of $D$ except $c$. Then $D'$ is a connected checkerboard-colored link diagram with $n-1$ crossings and with associated Tait graph $G' \cong G/e$. See the top right of Figure~\ref{nug4}. Hence, by the inductive hypothesis and the observation that $G'|E_{\sigma'} \cong (G|E_{\sigma})/e$, we get that
\begin{eqnarray*}
\left|s_{\sigma}(D)\right| & = & \left|s_{\sigma'}(D')\right|\\
& = & e(G'|E_{\sigma'})-v(G'|E_{\sigma'})+2k(G'|E_{\sigma'})\\
& = & (e(G|E_{\sigma})-1)-(v(G|E_{\sigma})-1)+2k(G|E_{\sigma})\\
& = & e(G|E_{\sigma})-v(G|E_{\sigma})+2k(G|E_{\sigma}).
\end{eqnarray*} 

\noindent \underline{Subcase 2}: Suppose a local neighborhood of $c$ has the checkerboard coloring represented by the top left of Figure~\ref{nug4} and suppose $c$ is B-resolved by the state $\sigma$. Then, by Remark~\ref{edgelabel} and Definition~\ref{phi}, we have that $e$ is a $(+,B)$ edge of $G$ and, therefore, $e \notin E_{\sigma}$. Let $D'$ be the link diagram formed by B-resolving the crossing $c$ and let $\sigma'$ be the state $\sigma$ restricted to the crossings of $D$ except $c$. Then $D'$ is a connected checkerboard-colored link diagram with $n-1$ crossings and with associated Tait graph $G' \cong G-e$. See the bottom right of Figure~\ref{nug4}. Hence, by the inductive hypothesis and the observation that $G'|E_{\sigma'} \cong G|E_{\sigma}$, we get that
\begin{eqnarray*}
\left|s_{\sigma}(D)\right| & = & \left|s_{\sigma'}(D')\right|\\
& = & e(G'|E_{\sigma'})-v(G'|E_{\sigma'})+2k(G'|E_{\sigma'})\\
& = & e(G|E_{\sigma})-v(G|E_{\sigma})+2k(G|E_{\sigma}).
\end{eqnarray*}   

\noindent \underline{Subcase 3}: Suppose a local neighborhood of $c$ has the checkerboard coloring represented by the bottom left of Figure~\ref{nug4} and suppose $c$ is A-resolved by the state $\sigma$. This subcase is very similar to Subcase~2.


\bigskip

\noindent \underline{Subcase 4}: Suppose a local neighborhood of $c$ has the checkerboard coloring represented by the bottom left of Figure~\ref{nug4} and suppose $c$ is B-resolved by the state $\sigma$. This subcase is very similar to Subcase~1.


\bigskip

\noindent \underline{Case 2}: Suppose $c$ is a nugatory crossing of $D$. This case is similar to Case 1. 
\end{proof}

\begin{proof}[Proof of Conclusion (2)]
By Proposition~\ref{adeqaltprop}, $D$ is $\sigma$-adequate with respect to a state $\sigma$ if and only if $\left|s_{\sigma}(D_{c}')\right|<\left|s_{\sigma}(D)\right|$, where $D_{c}'$ is the link diagram obtained from $D$ by changing the crossing type of a single crossing $c$. By Definition~\ref{Tait}, the crossing $c$ of $D$ corresponds to an edge $e$ of the Tait graph $G$ and changing the crossing type of $c$ to form $D_{c}'$ corresponds to changing the edge sign of $e$. Let $G'$ denote the Tait graph associated to $D_{c}'$ and let $e'$ denote the edge of $G'$ corresponding to the edge $e$ of $G$ with its sign changed. Since the resolution type of $c$ is not changed when forming $D_{c}'$ from $D$, then the state $\sigma$ of $D$ can be identified with a state, call it $\sigma$, of $D_{c}'$. By Conclusion~(1), the inequality $\left|s_{\sigma}(D_{c}')\right|<\left|s_{\sigma}(D)\right|$ is equivalent to the inequality
\begin{equation}
e(G'|E_{\sigma})-v(G'|E_{\sigma})+2k(G'|E_{\sigma}) < e(G|E_{\sigma})-v(G|E_{\sigma})+2k(G|E_{\sigma}).
\label{stateinequal}
\end{equation}

\noindent \underline{Case 1}: Suppose $e$ is a $(+,A)$ edge of $G$. Then $e \in E_{\sigma}$ is an edge of $G|E_{\sigma}$ while $e' \notin E_{\sigma}$ is not an edge of $G'|E_{\sigma}$. Therefore, since $G'|E_{\sigma} \cong (G|E_{\sigma})-e$, then $$e(G'|E_{\sigma})-v(G'|E_{\sigma})+2k(G'|E_{\sigma})=(e(G|E_{\sigma})-1)-v(G|E_{\sigma})+2k((G|E_{\sigma})-e).$$ By substituting this information into Inequality~\ref{stateinequal} and canceling and rearranging terms, we get that $k((G|E_{\sigma})-e)-k(G|E_{\sigma}) < \dfrac{1}{2}$. Since we are working with integers, this inequality is equivalent to $k((G|E_{\sigma})-e)-k(G|E_{\sigma}) \leq 0$, which is equivalent to $k((G|E_{\sigma})-e) \leq k(G|E_{\sigma})$. Since deleting an edge either leaves the number of connected components unchanged or increases the number of connected components by one, then the previous inequality is equivalent to the condition that $k((G|E_{\sigma})-e) \neq k(G|E_{\sigma})+1$. This says that deleting $e$ from $G|E_{\sigma}$ does not increase the number of connected components, which means that $e \in E_{\sigma}$ does not correspond a bridge of $G|E_{\sigma}$. 

\bigskip

\noindent \underline{Case 2}: Suppose $e$ is a $(+,B)$ edge of $G$. Then $e \notin E_{\sigma}$ is not an edge of $G|E_{\sigma}$ while $e' \in E_{\sigma}$ is an edge of $G'|E_{\sigma}$. Therefore, since $G'|E_{\sigma} \cong (G|E_{\sigma}) \cup e$, then $$e(G'|E_{\sigma})-v(G'|E_{\sigma})+2k(G'|E_{\sigma})=(e(G|E_{\sigma})+1)-v(G|E_{\sigma})+2k((G|E_{\sigma}) \cup e).$$ By substituting this information into Inequality~\ref{stateinequal} and canceling and rearranging terms, we get that $k((G|E_{\sigma}) \cup e)-k(G|E_{\sigma}) < -\dfrac{1}{2}$. Since we are working with integers, this inequality is equivalent to $k((G|E_{\sigma}) \cup e)-k(G|E_{\sigma}) \leq -1$, which is equivalent to $k((G|E_{\sigma}) \cup e) \leq k(G|E_{\sigma})-1$. Since adding an edge either leaves the number of connected components unchanged or decreases the number of connected components by one, then the previous inequality is equivalent to the condition that $k((G|E_{\sigma}) \cup e) = k(G|E_{\sigma})-1$. This says that adding $e$ to $G|E_{\sigma}$ decreases the number of connected components, which says that $e$ is a bridge of $(G|E_{\sigma}) \cup e$. This means that $e$ is not contained in a cycle of $(G|E_{\sigma}) \cup e$, which means that $e \notin E_{\sigma}$ does not correspond to a loop of $G/E_{\sigma}$. 

\bigskip

\noindent \underline{Case 3}: Suppose $e$ is a $(-,A)$ edge of $G$. This case is very similar to Case~2. 

\bigskip

\noindent \underline{Case 4}: Suppose $e$ is a $(-,B)$ edge of $G$. This case is very similar to Case~1.

\bigskip

Since Case~1 and Case~4 show that $G|E_{\sigma}$ contains no bridges and since Case~2 and Case~3 show that $G/E_{\sigma}$ contains no loops, then we have the desired result.  
\end{proof}

From Proposition~\ref{circleprop}, we are able to prove Theorem~\ref{introthm} from the introduction, which gives a characterization of $\sigma$-adequacy for a link diagram that uses only the edge-restricted Tait graph $G|E_{\sigma}$. This theorem is the first main component of a method, presented in Section~\ref{methodsec}, to find all of the $\sigma$-adequate states of a link diagram. 

\begin{proof}[Proof of Theorem~\ref{introthm}]
By Conclusion~(2) of Proposition~\ref{circleprop}, $D$ is $\sigma$-adequate with respect to a state $\sigma$ if and only if $G|E_{\sigma}$ contains no bridges and $G/E_{\sigma}$ contains no loops. Since an edge of a graph is not a bridge if and only if that edge is contained in a cycle of the given graph, then the condition that $G|E_{\sigma}$ contains no bridges is equivalent to the condition that every edge of $G|E_{\sigma}$, which can be identified with an edge in the set $E_{\sigma} \subseteq E(G)$, is contained in a cycle of $G|E_{\sigma}$. 

Since $D$ is reduced, then, by Remark~\ref{nugatoryrmk}, $G$ contains no loops. This means that loops of $G/E_{\sigma}$ must correspond to certain non-loop edges of $G$. Since a connected component of the subgraph $G|E_{\sigma}$ of $G$ becomes a single vertex of $G/E_{\sigma}$ when we contract the edges of $E_{\sigma}$ to form $G/E_{\sigma}$ from $G$, then an edge of $\overline{E_{\sigma}}=E(G)-E_{\sigma}$ with both endpoints on a connected component of $G|E_{\sigma}$ becomes a loop in $G/E_{\sigma}$. Conversely, since a loop of $G/E_{\sigma}$ must come from an edge $e \in \overline{E_{\sigma}}$ in a cycle of the subgraph $(G|E_{\sigma}) \cup e$ of $G$, then $e$ must have both endpoints on a connected component of $G|E_{\sigma}$. Therefore, the condition that $G/E_{\sigma}$ contains no loops is equivalent to the condition that no edge of $\overline{E_{\sigma}}$ has both endpoints on a connected component of $G|E_{\sigma}$. 
\end{proof}

\subsection{Generalizing Thistlethwaite's Nonvanishing Theorem}

We are now ready to prove Theorem~\ref{mainthm1} from the introduction, which shows that a connected checkerboard-colored link diagram $D$ is $\sigma$-adequate with respect to a state $\sigma$ if and only if $\phi_{D}^{\sigma}(t) \neq 0$. 

\begin{proof}[Proof of Theorem~\ref{mainthm1}]
By Conclusion~(2) of Proposition~\ref{circleprop}, $D$ is $\sigma$-adequate with respect to a state $\sigma$ if and only if $G|E_{\sigma}$ contains no bridges and $G/E_{\sigma}$ contains no loops. By Property~(2) of Proposition~\ref{phiprop}, we have that $G|E_{\sigma}$ contains no bridges and $G/E_{\sigma}$ contains no loops if and only if $\phi_{D}^{\sigma}(t) \neq 0$. 
\end{proof}

\subsection{A Sigma-Adequate State Sum for the Symmetrized Tutte Polynomial}

We now define the Tutte polynomial that will be expressed as a $\sigma$-adequate state sum when applied to the Tait graph associated to a link diagram. 

\begin{definition} 
If $\chi_{G}(t,t)$ denotes the evaluation of the Tutte polynomial $\chi_{G}(x,y)$ of a graph $G$ at $x=y=t$, then we call $\chi_{G}(t,t)$ the \textit{symmetrized Tutte polynomial} of $G$. 
\end{definition}

\begin{remark} \label{anychecker}
By Remark~\ref{planardualrmk}, changing the checkerboard coloring of $D$ corresponds to changing from the associated Tait graph $G$ to its planar dual $G^{*}$. Moreover, by Proposition~\ref{planardualprop}, the variables $x$ and $y$ of the Tutte polynomial reverse roles under the operation of planar duality, that is, $\chi_{G^{*}}(x,y)=\chi_{G}(y,x)$. By combining these two ideas, we get that the initial choice of checkerboard coloring of $D$ has no effect on the symmetrized Tutte polynomial $\chi_{G}(t,t)$ of $G$.
\end{remark}

By restricting Theorem~1 of \cite{Kook} to graphs (a special type of matroid) and making the substitution $x=y=t$, we get the following result. 

\begin{theorem}{(Special Case of Theorem 1 of \cite{Kook})} \label{Tuttesum}
The symmetrized Tutte polynomial $\chi_{G}(t, t)$ of a graph $G$ can be written as the summation $\displaystyle \chi_{G}(t, t)=\sum_{H\subseteq E(G)}\ \chi_{G|H}(0, t)\cdot\chi_{G/H}(t, 0)$. \qed
\end{theorem}

We now prove Theorem~\ref{mainthm2} from the introduction, which shows that the symmetrized Tutte polynomial of the Tait graph $G$ associated to a checkerboard-colored link diagram $D$ can be expanded as a sum of the $\phi_{D}^{\sigma}(t)$ polynomials over the $\sigma$-adequate states of $D$. 

\begin{proof}[Proof of Theorem~\ref{mainthm2}]
Let $e$ denote an arbitrary edge of $G$ and let $c$ denote the corresponding crossing of $D$. Note that the $+$ and $-$ signs on the edges of $G$ have already been determined by the choice of checkerboard coloring for $D$. Thus, for each choice of edge subset $H \subseteq E(G)$ (which also forces a choice of edge subset $\overline{H}=E(G)-H \subseteq E(G)$), we have both a corresponding choice of state $\sigma$ of $D$ and a corresponding choice of edge subset $E_{\sigma}\subseteq E(G)$. Specifically, we make the $+$ edges of $H$ correspond to A-resolved crossings of $D$, the $-$ edges of $H$ correspond to B-resolved crossings of $D$, the $+$ edges of $\overline{H}$ correspond to B-resolved crossings of $D$, and the $-$ edges of $\overline{H}$ correspond to A-resolved crossings of $D$. See Table~\ref{edgetable} for a depiction of this correspondence.    

By Theorem~\ref{Tuttesum}, Table~\ref{edgetable}, and Definition~\ref{phi}, we get that 
\begin{eqnarray*}
\displaystyle \chi_{G}(t, t) & = & \sum_{H\subseteq E(G)}\ \chi_{G|H}(0, t)\cdot\chi_{G/H}(t, 0)\\
& = & \sum_{E_{\sigma}\subseteq E(G)}\ \chi_{G|E_{\sigma}}(0, t)\cdot\chi_{G/E_{\sigma}}(t, 0)\\
& = & \sum_{E_{\sigma}\subseteq E(G)}\ \phi_{D}^{\sigma}(t)\\
& = & \sum_{\text{states}\ \sigma}\ \phi_{D}^{\sigma}(t).
\end{eqnarray*} 
By Theorem~\ref{mainthm1}, since $D$ is $\sigma$-adequate with respect to a state $\sigma$ if and only if $\phi_{D}^{\sigma}(t) \neq 0$, then the summation above can be written as $$\displaystyle \chi_{G}(t, t)=\sum_{\sigma\text{-adequate\ states}\ \sigma}\ \phi_{D}^{\sigma}(t).$$ \end{proof}

\begin{table}
\begin{center}
\begin{tabular}{c|c|c|c} 
sign of edge $e$ & $e \in H \subseteq E(G)$? & resolution type of crossing $c$ & $e \in E_{\sigma}\subseteq E(G)$?\\
\hline
$+$ & Yes & A & Yes\\
\hline
$-$ & Yes & B & Yes\\
\hline
$+$ & No & B & No\\
\hline
$-$ & No & A & No\\
\end{tabular}\end{center}
\caption{A table relating subsets $H \subseteq E(G)$ to both states $\sigma$ of $D$ and subsets $E_{\sigma}\subseteq E(G)$.}
\label{edgetable}
\end{table}

From Theorem~\ref{mainthm2}, we are able to prove Theorem~\ref{maincor1} from the introduction, which provides upper and lower bounds on the number of $\sigma$-adequate states of a given link diagram. 

\begin{proof}[Proof of Theorem~\ref{maincor1}]
Let $N$ denote the number of $\sigma$-adequate states of $D$. By Proposition~\ref{checkerrmk}, since the checkerboard states of $D$ are always $\sigma$-adequate states, then we have that $N \geq 2$. By the spanning tree definition of the Tutte polynomial $\chi_{G}(x,y)$, we have that $$\displaystyle \chi_{G}(t,t)=\sum_{\text{spanning\ trees}\ T\ \text{of}\ G} t^{IA(T)+EA(T)},$$ where $IA(T)$ denotes the internal activity of $T$ and $EA(T)$ denotes the external activity of $T$. (For more information on the spanning tree expansion of the Tutte polynomial, refer to \cite{Thistlethwaite1}.) By Theorem~\ref{mainthm2}, we have that $$\displaystyle \chi_{G}(t,t)=\sum_{\sigma\text{-adequate\ states}\ \sigma\ \text{of}\ D}\phi_{D}^{\sigma}(t).$$ Expand $\chi_{G}(t,t)$ as a polynomial $\chi_{G}(t,t)=a_{m}t^{m}+a_{m-1}t^{m-1}+\cdots+a_{1}t+a_{0}$. Since, by Property~(1) of Proposition~\ref{phiprop}, the coefficients of $\phi_{D}^{\sigma}(t)$ are nonnegative, then $a_{i} \geq 0$ for all $0 \leq i \leq m$ and no cancellation can occur in the $\sigma$-adequate state sum expansion above. Since $\phi_{D}^{\sigma}(t)$ is a polynomial in $t$, then it is potentially possible that each of the polynomials $\phi_{D}^{\sigma}(t)$ corresponding to a $\sigma$-adequate state $\sigma$ of $D$ is a single monomial $t^{IA(T)+EA(T)}$ for some spanning tree $T$ of $G$. Otherwise, some of the polynomials $\phi_{D}^{\sigma}(t)$ are sums of monomials from the spanning tree expansion above. Therefore, we have that $N \leq a_{m}+a_{m-1}+\cdots+a_{1}+a_{0}=\chi_{G}(1,1)$, where $\chi_{G}(1,1)$ is the number of spanning trees of $G$. 
\end{proof}

To show that the lower bound of Theorem~\ref{maincor1} is sharp for diagrams of an infinite family of links, we prove that the standard diagrams of the $(2,n)$-torus links, for $n \geq 2$, have exactly two $\sigma$-adequate states.  

\begin{proposition} \label{torusprop}
Let $D$ be the standard diagram of the $(2,n)$-torus link for $n \geq 2$. Then $D$ has exactly two $\sigma$-adequate states, namely the black and white checkerboard states $\sigma_{bl}$ and $\sigma_{wh}$ of $D$.  
\end{proposition}


\begin{proof}
Without loss of generality, give $D$ the checkerboard coloring where the unbounded region is shaded white. By Definition~\ref{Tait}, the corresponding Tait graph $G$ is the cycle graph $C_{n}$ with $n$ edges. By Proposition~\ref{signlemma} and Proposition~\ref{Tuttecycleprop}, for the black and white checkerboard states $\sigma_{bl}$ and $\sigma_{wh}$ of $D$, we have that $\phi_{D}^{\sigma_{bl}}(t)=\chi_{C_{n}}(t,0)=t^{n-1}+t^{n-2}+\cdots+t^{2}+t$ and $\phi_{D}^{\sigma_{wh}}(t)=\chi_{C_{n}}(0,t)=t$. Since these polynomials sum to give $\chi_{C_{n}}(t,t)=t^{n-1}+t^{n-2}+\cdots+t^{2}+2t$, then, by Theorem~\ref{mainthm1} and Theorem~\ref{mainthm2}, there are exactly two $\sigma$-adequate states of $D$. 
\end{proof}

To show that the upper bound of Theorem~\ref{maincor1} is sharp for diagrams of an infinite family of links, we prove that the number of $\sigma$-adequate states of the connect sum of $n \geq 1$ standard diagrams of the $(2,2)$-torus link (the \textit{Hopf link}) is exactly the number of spanning trees in either of its Tait graphs. 

\begin{proposition} \label{Hopfprop}
Let $D$ be the connect sum of $n \geq 1$ standard diagrams of the $(2,2)$-torus link (the Hopf link). Then the number of $\sigma$-adequate states of $D$ is exactly the number of spanning trees in a Tait graph $G$ of $D$.  
\end{proposition}

\begin{proof}
Without loss of generality, give $D$ the checkerboard coloring where the unbounded region is shaded white. By Definition~\ref{Tait}, the corresponding Tait graph $G$ is a path of $n$ double edges. Therefore, by Definition~\ref{Tuttepoly}, we get that $\chi_{G}(t,t)=(2t)^{n}$, which gives that the number of spanning trees of $G$ is $\chi_{G}(1,1)=2^{n}$. By Figure~\ref{res} it can be seen that, for $D$ to be $\sigma$-adequate, the pair of crossings in a $T(2,2)$ connect summand of $D$ (in a \textit{twist region} of $D$) must either be both A-resolved or both B-resolved. The converse of this statement can also be shown to be true. Therefore, since there are $n$ connect summands and a choice of either both A-resolutions or both B-resolutions for each connect summand, then there are exactly $2^{n}$ $\sigma$-adequate states of $D$.  
\end{proof}



\subsection{A Method for Finding All Sigma-Adequate States of a Link Diagram} \label{methodsec}

We now combine results to provide a method for finding all of the $\sigma$-adequate states of a connected, reduced, checkerboard-colored link diagram $D$. 

\bigskip

\noindent \textbf{\underline{Step 1}:} Construct the Tait graph $G$ associated to the checkerboard-colored link diagram $D$. 

\bigskip

\noindent \textbf{\underline{Step 2}:} Look for partitions $E(G)=E_{\sigma} \sqcup \overline{E_{\sigma}}$ of the edges of $G$ that satisfy Condition~(1) and Condition~(2) of Theorem~\ref{introthm}. 

\bigskip

\noindent \textbf{\underline{Step 3}:} Given the collection of edge partitions $E(G)=E_{\sigma} \sqcup \overline{E_{\sigma}}$ from Step~2, use Table~\ref{edgetable} to find the corresponding $\sigma$-adequate states $\sigma$. 

\bigskip

\noindent \textbf{\underline{Step 4}:} Use Definition~\ref{phi} to compute the $\phi_{D}^{\sigma}(t)$ polynomials for the $\sigma$-adequate states found in Step 3. Note that, by Theorem~\ref{mainthm1}, getting that $\phi_{D}^{\sigma}(t) \neq 0$ will confirm that $D$ is $\sigma$-adequate with respect to the state $\sigma$.  

\bigskip

\noindent \textbf{\underline{Step 5}:} Use Definition~\ref{Tuttepoly} to compute the symmetrized Tutte polynomial $\chi_{G}(t,t)$. 

\bigskip

\noindent \textbf{\underline{Step 6}:} Find the sum of the $\phi_{D}^{\sigma}(t)$ polynomials from Step 4 and compare this to the symmetrized Tutte polynomial $\chi_{G}(t,t)$. If equality is achieved, then, by Theorem~\ref{mainthm1} and Theorem~\ref{mainthm2}, all of the $\sigma$-adequate states of $D$ have been found (and confirmed). If equality is not achieved, return to Step~2. 

\bigskip

\noindent \textbf{\underline{An Example of the Method}:} Consider the connected reduced diagram, call it $D$, of the nonalternating and \textit{inadequate} (neither A- nor B-adequate) knot $11n_{95}$ from the KnotInfo website (\cite{KnotInfo}). Without loss of generality, checkerboard color $D$ so that the unbounded region of $D$ is shaded white. This is sometimes called the \textit{canonical checkerboard coloring}. 

\newpage

\noindent \underline{Step 1}: From our checkerboard coloring of $D$, we get the Tait graph $G$ depicted on the left side of Figure~\ref{11n95}. 

\bigskip

We now break Step~2 into two cases. This division depends on whether or not $G|E_{\sigma}$ consists entirely of \textit{fundamental cycles} of $G$, were a \textit{fundamental cycle} of the planar graph $G$ is the boundary of a face (complementary region) of $G$. 

\bigskip

\noindent \underline{Step 2A}: First, consider the case of Theorem~\ref{introthm} where $G|E_{\sigma}$ is a union of isolated vertices of $G$ and fundamental cycles of $G$ and where no edge of $\overline{E_{\sigma}}$ has both endpoints on a connected component of $G|E_{\sigma}$. For our labeling of the fundamental cycles of $G$, see the center of Figure~\ref{11n95}. See Table~\ref{statetable} for a list of the 19 collections of fundamental cycles of $G$ with edge sets $E_{\sigma}$. 

\bigskip

\noindent \underline{Step 2B}: Second, consider the case of Theorem~\ref{introthm} where $G|E_{\sigma}$ is a union of isolated vertices of $G$ and cycles of $G$ that are not all fundamental cycles and where no edge of $\overline{E_{\sigma}}$ has both endpoints on a connected component of $G|E_{\sigma}$. In the case of the standard diagram $D$ of the knot $11n_{95}$, there is only one such union of cycles. See the right side of Figure~\ref{11n95}. 

\bigskip

\noindent \underline{Step 3A}: Using Table~\ref{edgetable}, the edge sets $E_{\sigma}$ from Step~2A correspond to the black checkerboard state $\sigma_{bl}$, the white checkerboard state $\sigma_{wh}$, and 17 other $\sigma$-adequate states of $D$, which we label as $\sigma_{1}, \sigma_{2}, \ldots, \sigma_{17}$. As an aside, the state $\sigma_{2}$ is the Seifert state of $D$. See Table~\ref{statetable} for the list of the 19 $\sigma$-adequate states $\sigma$ corresponding to the edge sets $E_{\sigma}$ found in Step~2A. 

\bigskip

\noindent \underline{Step 3B}: Using Table~\ref{edgetable}, the edge set $E_{\sigma}$ from Step~2B corresponds to a $\sigma$-adequate state, call it $\widetilde{\sigma}$, of $D$. 

\bigskip

\noindent \underline{Step 4A}: Using Definition~\ref{phi}, we compute the polynomials $\phi_{D}^{\sigma}(t)$ for the 19 $\sigma$-adequate states found in Step~3A. See Table~\ref{statetable}. Since $\phi_{D}^{\sigma}(t) \neq 0$ for each such state, then we have confirmation that $D$ is $\sigma$-adequate with respect to each state in this collection of states.  

\bigskip

\noindent \underline{Step 4B}: Using Definition~\ref{phi}, we get that $\phi_{D}^{\widetilde{\sigma}}(t) = t^{4}+t^{3}$. Since $\phi_{D}^{\widetilde{\sigma}}(t) \neq 0$, then we have confirmation that $D$ is $\sigma$-adequate with respect to the state $\widetilde{\sigma}$.

\bigskip

\noindent \underline{Step 5}: Using Definition~\ref{Tuttepoly}, we compute the symmetrized Tutte polynomial of $G$ to be
\begin{equation}
\chi_{G}(t,t)=2t^{6}+16t^{5}+48t^{4}+62t^{3}+33t^{2}+6t.
\label{chieqn}
\end{equation}

\noindent \underline{Step 6}: By summing the polynomials $\phi_{D}^{\sigma}(t)$ for the states $\sigma_{bl}$, $\sigma_{wh}$, and $\sigma_{1}, \sigma_{2}, \ldots, \sigma_{17}$ from Step~4A (see See Table~\ref{statetable}) and adding this to $\phi_{D}^{\widetilde{\sigma}}(t) = t^{4}+t^{3}$ from Step~4B, we get the symmetrized Tutte polynomial $\chi_{G}(t,t)$, as given in Equation~\ref{chieqn}. 

\bigskip

To summarize, we have found (and confirmed) all 20 $\sigma$-adequate states of the standard diagram $D$ of the knot $11n_{95}$. The advantage of this method is that it only requires an investigation of the Tait graph $G$ to find $\sigma$-adequate states of the link diagram $D$ and utilizes computations of Tutte polynomials for the Tait graph $G$, the edge-restricted Tait graphs $G|E_{\sigma}$, and the edge-contracted Tait graphs $G/E_{\sigma}$ to confirm that all of the $\sigma$-adequate states of $D$ have been found. 

\begin{figure}
\centering
\def\svgwidth{5in}
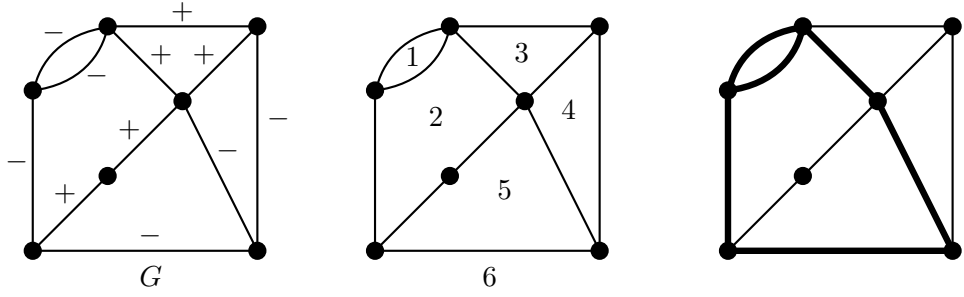
\caption{The Tait graph $G$ associated to the canonical checkerboard coloring of the standard diagram $D$ of the knot $11n_{95}$ (left), a labeling of the fundamental cycles of $G$ (center), and the cycles of $G$ whose edges form the edge set $E_{\widetilde{\sigma}}$ corresponding to the $20^{\text{th}}$ $\sigma$-adequate state $\widetilde{\sigma}$ of $D$. }
\label{11n95}
\end{figure}
\begin{table}
\begin{center}
\begin{tabular}{c|c|c} 
$\sigma$-adequate states $\sigma$ & fundamental cycle(s) of $G$ & polynomials $\phi_{D}^{\sigma}(t)$\\
 & with edge set $E_{\sigma}$ & \\
\hline
$\sigma_{bl}$ & NONE & $t^{6}+4t^{5}+8t^{4}+11t^{3}+9t^{2}+3t$\\
\hline
$\sigma_{1}$ & $1$ & $t^{6}+4t^{5}+8t^{4}+8t^{3}+3t^{2}$\\
\hline
$\sigma_{2}$ & $3$ & $t^{5}+2t^{4}+3t^{3}+t^{2}$\\
\hline
$\sigma_{3}$ & $4$ & $t^{5}+2t^{4}+2t^{3}+t^{2}$\\
\hline
$\sigma_{4}$ & $5$ & $t^{4}+2t^{3}+t^{2}$\\
\hline
$\sigma_{5}$ & $1,2$ & $t^{4}+2t^{3}+t^{2}$\\
\hline
$\sigma_{6}$ & $1,3$ & $t^{5}+2t^{4}+t^{3}$\\
\hline
$\sigma_{7}$ & $1,4$ & $t^{5}+2t^{4}+t^{3}$\\
\hline
$\sigma_{8}$ & $1,5$ & $t^{4}+t^{3}$\\
\hline
$\sigma_{9}$ & $1,6$ & $t^{4}+2t^{3}+t^{2}$\\
\hline
$\sigma_{10}$ & $3,4$ & $t^{5}+3t^{4}+3t^{3}+t^{2}$\\
\hline
$\sigma_{11}$ & $4,5$ & $t^{4}+2t^{3}+t^{2}$\\
\hline
$\sigma_{12}$ & $1,2,3$ & $t^{4}+2t^{3}+t^{2}$\\
\hline
$\sigma_{13}$ & $1,2,5$ & $t^{4}+2t^{3}+t^{2}$\\
\hline
$\sigma_{14}$ & $1,3,4$ & $t^{5}+2t^{4}+t^{3}$\\
\hline
$\sigma_{15}$ & $1,4,5$ & $t^{4}+t^{3}$\\
\hline
$\sigma_{16}$ & $3,4,5$ & $t^{4}+2t^{3}+t^{2}$\\
\hline
$\sigma_{17}$ & $1,3,4,6$ & $t^{5}+4t^{4}+5t^{3}+2t^{2}$\\
\hline
$\sigma_{wh}$ & $1,2,3,4,5,6$ & $t^{5}+5t^{4}+10t^{3}+9t^{2}+3t$\\
\end{tabular}
\caption{A table depicting 19 of the $\sigma$-adequate states for the standard diagram of the knot $11n_{95}$ (left), the fundamental cycles of $G$ whose edges form the 19 corresponding edge subsets $E_{\sigma} \subseteq E(G)$ (center), and the 19 corresponding polynomials $\phi_{D}^{\sigma}(t)$ (right).}
\label{statetable}
\end{center}
\end{table}

\section{Applying the Perspective of Ribbon Graphs and Partial Duality} \label{duality}
\label{PDsec}

In this section, we use work of Chmutov on ribbon graphs and partial duality (\cite{Chmutov}) to provide an alternate perspective on $\sigma$-adequacy for link diagrams.

\subsection{Ribbon Graphs, States and Partial Duality}

We begin by introducing the notions of ribbon graphs and partial duality, as defined by Chmutov in \cite{Chmutov}.

\begin{definition} \label{ribbon}
A \textit{ribbon graph} $\mathbb{G}$ is a surface with boundary that consists of a vertex set $V(\mathbb{G})$ and an edge set $E(\mathbb{G})$, where the \textit{vertices} are a collection of disks, where the \textit{edges} are a collection of bands, and where the following conditions hold. 
\begin{itemize}
	\item The vertices and edges intersect in a disjoint collection of line segments.
	\item Each such line segment is formed by the intersection of one vertex and one edge. 
	\item Each edge contains exactly two such line segments. 
\end{itemize}
To distinguish ribbon graphs from graphs, we will often use the term \textit{vertex disk} instead of vertex and \textit{edge ribbon} instead of edge. For examples of ribbon graphs, see Figure~\ref{PDExample1} and Figure~\ref{PDExample2}. 
\end{definition}

As shown in \cite{Chmutov}, to each state $\sigma$ of a link diagram, we can associate a \textit{state ribbon graph} $\mathbb{G}_{\sigma}$. 

\begin{definition} \label{ribbondef}
Given a state $\sigma$ of a link diagram $D$, we construct the corresponding \textit{state ribbon graph} $\mathbb{G}_{\sigma}$ as follows. Recall that, by Definition~\ref{trivalent}, we can use the state $\sigma$ to construct the state graph $H_{\sigma}$, which consists of a disjoint collection of state circles and a disjoint collection of state segments. By capping off each state circle of $H_{\sigma}$ with a disk (in such a way that disks coming from inner circles lie above disks coming from outer circles), we obtain the disjoint collection of vertex disks of $\mathbb{G}_{\sigma}$. By covering each state segment of $H_{\sigma}$ with a planar band that deformation retracts back to the state segment, we obtain the collection of edge ribbons of $\mathbb{G}_{\sigma}$. It can quickly be seen that the conditions required by Definition~\ref{ribbon} are satisfied. 
\end{definition}


Given the state ribbon graph $\mathbb{G}_{\sigma}$ of a link diagram $D$, we can reinterpret the $\sigma$-adequacy of $D$ as follows. 

\begin{remark} \label{adeqrmk}
By Definition~\ref{adeq}, we have that a link diagram $D$ is $\sigma$-adequate with respect to a state $\sigma$ if and only if the state graph $H_{\sigma}$ contains no state segments that connect a state circle to itself. Since the state circles of $H_{\sigma}$ correspond to the vertex disks of $\mathbb{G}_{\sigma}$ and since the state segments of $H_{\sigma}$ correspond to the edge ribbons of $\mathbb{G}_{\sigma}$, then we have that $D$ is $\sigma$-adequate with respect to a state $\sigma$ if and only if the state ribbon graph $\mathbb{G}_{\sigma}$ associated to the state $\sigma$ contains no loops.
\end{remark}

Recall that, by Definition~\ref{planardualdef}, we can form the planar dual $G^{*}$ of a planar graph $G$. By extending the notion of planar duality to a (not necessarily planar) ribbon graph $\mathbb{G}$, we can define the \textit{geometric dual} $\mathbb{G}^{*}$ of $\mathbb{G}$. By applying geometric duality to a subset, call it $F$, of the edge ribbons of a ribbon graph $\mathbb{G}$, we can (roughly speaking) form the \textit{partial dual} $\mathbb{G}^F$ of $\mathbb{G}$. We define one case of partial duality below. See \cite{Chmutov} for full details about geometric and partial duality. 

\begin{definition} \label{PDdef}
Let $\mathbb{G}$ be a ribbon graph and let $e$ be a non-loop edge ribbon of $\mathbb{G}$. Recall that the edge ribbon $e$ is a four-sided band and notice that two opposite sides of the edge ribbon $e$ run along boundaries of vertex disks and the remaining two opposite sides of the edge ribbon $e$ join the distinct vertex disk endpoints of $e$. The \textit{partial dual of $\mathbb{G}$ with respect to the non-loop edge ribbon $e$}, denoted $\mathbb{G}^{e}$ and formed by \textit{dualizing the non-loop edge ribbon $e$}, is the ribbon graph that results from reversing the roles of the pairs of opposite sides of the edge ribbon $e$ of $\mathbb{G}$. See Figure~\ref{PD} for a local depiction of partial duality with respect to a non-loop edge ribbon. Note that the vertex disk structure of the ribbon graph is changed by the operation of partial duality. Given a collection $F=\left\{e_{1}, e_{2}, \ldots, e_{k}\right\}$ of non-loop edge ribbons of $\mathbb{G}$, the \textit{partial dual of $\mathbb{G}$ with respect to the collection $F$ of non-loop edge ribbons}, denoted $\mathbb{G}^{F}$ and formed by \textit{dualizing the collection $F$ of non-loop edge ribbons}, is the ribbon graph that results from dualizing the edge ribbons of $F$ all at once. See Figure~\ref{PDExample1} for a depiction of the construction of the partial dual of the cycle ribbon graph $\mathbb{C}_{4}$ with respect to all four of its edge ribbons.
\end{definition}

\begin{figure} 
\centering
\def\svgwidth{4in}
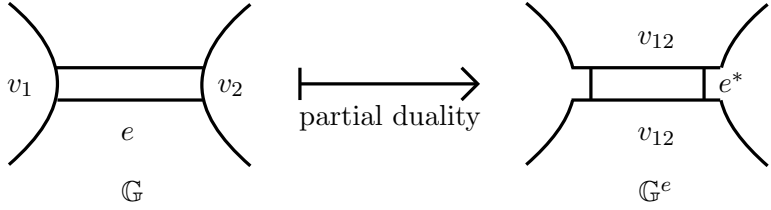
\caption{A depiction of the construction of the partial dual, $\mathbb{G}^{e}$, of $\mathbb{G}$ with respect to a non-loop edge ribbon $e$.}
\label{PD}
\end{figure}

\begin{figure} 
\centering
\def\svgwidth{4in}
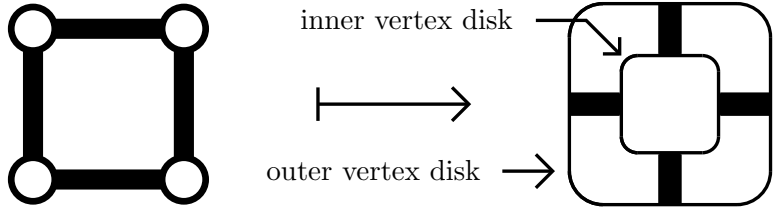
\caption{A depiction of the construction of the partial dual of the cycle ribbon graph $\mathbb{C}_{4}$ with respect to all four of its edge ribbons. The edge ribbons both before and after partial duality are shaded black while the vertex disks both before and after partial duality are shaded white. We view inner vertex disks as lying above outer vertex disks in order to avoid any self intersections in the resulting ribbon graph.}
\label{PDExample1}
\end{figure}

Using the notion of partial duality, Chmutov shows that, given two states of a link diagram, the corresponding state ribbon graphs are partial duals of each other (Lemma~6.2 of \cite{Chmutov}). By applying this result to the black checkerboard state $\sigma_{bl}$ and any other state $\sigma$ of a checkerboard-colored link diagram $D$, we get the following important result.

\begin{proposition}{(Special Case of Lemma~6.2 of \cite{Chmutov})} \label{dualcor}
Let $\sigma$ be a state of a checkerboard-colored link diagram $D$ and let $\sigma_{bl}$ denote the black checkerboard state of $D$. Then the state ribbon graph $\mathbb{G}_{\sigma}$ is a partial dual of the black checkerboard state ribbon graph $\mathbb{G}_{\sigma_{bl}}$. In particular, $\mathbb{G}_{\sigma}$ and $\mathbb{G}_{\sigma_{bl}}$ are partially dual with respect to the set of edge ribbons corresponding to the crossings of $D$ where the states $\sigma$ and $\sigma_{bl}$ differ from each other. \qed 
\end{proposition}

\subsection{Relating the Tait Graph to the State Ribbon Graphs of a Link Diagram}

We now compare the Tait graph $G$ associated to a checkerboard-colored link diagram $D$ to the state ribbon graph $\mathbb{G}_{\sigma_{bl}}$ associated to the black checkerboard state $\sigma_{bl}$ of $D$. As will be shown below, these graphs encode the same information. 

\begin{proposition} \label{Gprop}
Let $D$ be a checkerboard-colored link diagram with associated Tait graph $G$, let $\sigma_{bl}$ denote the black checkerboard state of $D$, and let $\mathbb{G}_{\sigma_{bl}}$ denote the corresponding black checkerboard state ribbon graph. Then $G$ is the spine of $\mathbb{G}_{\sigma_{bl}}$. Therefore, the cycles of $G$ correspond to the cycles of $\mathbb{G}_{\sigma_{bl}}$ and the planarity of $G$ corresponds to the planarity of $\mathbb{G}_{\sigma_{bl}}$.  
\end{proposition}

\begin{proof}
By Definition~\ref{Tait}, the vertices of $G$ correspond to the black regions of $D$. These regions correspond, by considering their boundaries, to the state circles of the black checkerboard state graph $H_{\sigma_{bl}}$, which bound the vertex disks of the black checkerboard state ribbon graph $\mathbb{G}_{\sigma_{bl}}$. Again by Definition~\ref{Tait}, the edges of $G$ correspond to the crossings of $D$ between black regions of $D$. These crossings correspond to the state segments of $H_{\sigma_{bl}}$, which are deformation retracts of the edge ribbons of $\mathbb{G}_{\sigma_{bl}}$. Thus, $G$ is the spine of $\mathbb{G}_{\sigma_{bl}}$. The remaining results follow immediately.
\end{proof}

\begin{remark} \label{nolooprmk}
For a reduced checkerboard-colored link diagram $D$ with associated Tait graph $G$, since $G$ is the spine of the black checkerboard state ribbon graph $\mathbb{G}_{\sigma_{bl}}$ and since $D$ is reduced, then, by Remark~\ref{nugatoryrmk}, $\mathbb{G}_{\sigma_{bl}}$ contains no loops. Hence, by Proposition~\ref{dualcor}, we can study the state ribbon graphs $\mathbb{G}_{\sigma}$ using partial duality for collections of non-loop edge ribbons of $\mathbb{G}_{\sigma_{bl}}$. 
\end{remark}

\begin{definition} \label{EPDdef}
Given a state $\sigma$ of a reduced checkerboard-colored link diagram $D$, let $E_{PD}^{\sigma}$ denote the collection of edge ribbons of $\mathbb{G}_{\sigma_{bl}}$ that must be dualized to realize $\mathbb{G}_{\sigma}$ as a partial dual of $\mathbb{G}_{\sigma_{bl}}$ and let $E_{F}^{\sigma}$ denote the collection of edge ribbons of $\mathbb{G}_{\sigma_{bl}}$ that are not dualized (are fixed) to realize $\mathbb{G}_{\sigma}$ as a partial dual of $\mathbb{G}_{\sigma_{bl}}$. By Remark~\ref{nolooprmk}, $\mathbb{G}_{\sigma_{bl}}$ has no loops. Therefore, by Definition~\ref{PDdef}, we have that $\mathbb{G}_{\sigma}=(\mathbb{G}_{\sigma_{bl}})^{E_{PD}^{\sigma}}$. Additionally, let $(E_{PD}^{\sigma})^{*}$ denote the collection of edge ribbons of $\mathbb{G}_{\sigma}$ that result from dualizing the edge ribbons of $E_{PD}^{\sigma}$ to realize $\mathbb{G}_{\sigma}$ as a partial dual of $\mathbb{G}_{\sigma_{bl}}$ and let $(E_{F}^{\sigma})^{*}$ denote the collection of edge ribbons of $\mathbb{G}_{\sigma}$ that result from not dualizing (fixing) the edge ribbons of $E_{PD}^{\sigma}$ to realize $\mathbb{G}_{\sigma}$ as a partial dual of $\mathbb{G}_{\sigma_{bl}}$. Note that the interiors of the edge ribbons of $(E_{F}^{\sigma})^{*}$ can be identified with the interiors of the edge ribbons of $E_{F}^{\sigma}$. The endpoints of these edge ribbons, however, may have been changed by the operation of partial duality. 
\end{definition}

Let $G$ denote the Tait graph associated to a checkerboard-colored link diagram $D$ and recall, by Definition~\ref{phi}, that $E_{\sigma}=E_{+}^{A} \cup E_{-}^{B} \subseteq E(G)$. In the following proposition, we show that the subset $E_{\sigma} \subseteq E(G)$ corresponds to the subset $E_{PD}^{\sigma} \subseteq E(\mathbb{G}_{\sigma_{bl}})$.  

\begin{proposition} \label{edgeprop}
Let $D$ be a connected, reduced, checkerboard-colored link diagram with at least one crossing and with associated Tait graph $G$. Let $\sigma$ be a state of $D$ with associated state ribbon graph $\mathbb{G}_{\sigma}$ and let $\sigma_{bl}$ denote the black checkerboard state of $D$ with associated black checkerboard state ribbon graph $\mathbb{G}_{\sigma_{bl}}$. Then the edges in the subset $E_{\sigma} \subseteq E(G)$ correspond bijectively to the edge ribbons in the subset $(E_{PD}^{\sigma})^{*} \subseteq E(\mathbb{G}_{\sigma})$, which correspond bijectively to the edge ribbons in the subset $E_{PD}^{\sigma} \subseteq E(\mathbb{G}_{\sigma_{bl}})$. As a result, we also get that the edges in the subset $\overline{E_{\sigma}} \subseteq E(G)$ correspond bijectively to the edge ribbons in the subset $(E_{F}^{\sigma})^{*} \subseteq E(\mathbb{G}_{\sigma})$, which correspond bijectively to the edge ribbons in the subset $E_{F}^{\sigma} \subseteq E(\mathbb{G}_{\sigma_{bl}})$. 
\end{proposition}

\begin{proof}
Given the Tait graph $G$ and the black checkerboard state $\sigma_{bl}$, Remark~\ref{edgelabel} allows us to classify the edges of $G$ as $(+,A)$, $(+,B)$, $(-,A)$, or $(-,B)$ edges. Since, by Proposition~\ref{Gprop}, $G$ is the spine of $\mathbb{G}_{\sigma_{bl}}$, then the edges of $G$ and $\mathbb{G}_{\sigma_{bl}}$ must be in bijective correspondence. Hence, by Proposition~\ref{signlemma}, every edge ribbon of the black checkerboard state ribbon graph $\mathbb{G}_{\sigma_{bl}}$ is either a $(+,B)$ edge or a $(-,A)$ edge. By Definition~\ref{EPDdef}, $E_{PD}^{\sigma}$ denotes the collection of edge ribbons of $\mathbb{G}_{\sigma_{bl}}$ that must be dualized to realize $\mathbb{G}_{\sigma}$ as a partial dual of $\mathbb{G}_{\sigma_{bl}}$. After applying the operation of partial duality to $\mathbb{G}_{\sigma_{bl}}$ to form the state ribbon graph $\mathbb{G}_{\sigma}$, the edge ribbons of $E_{PD}^{\sigma}$ become either $(+,A)$ edges or $(-,B)$ edges of $\mathbb{G}_{\sigma}$. Since, by Definition~\ref{phi}, we have that $E_{\sigma}=E_{+}^{A} \cup E_{-}^{B}$ denotes the set of $(+,A)$ and $(-,B)$ edges of $G$ corresponding to the state $\sigma$ of $D$, then we can see that the edges of $E_{\sigma} \subseteq E(G)$ correspond bijectively to the edge ribbons of $(E_{PD}^{\sigma})^{*} \subseteq E(\mathbb{G}_{\sigma})$. Since partial duality modifies but does not add or remove any edge ribbons, then the edge ribbons of $(E_{PD}^{\sigma})^{*} \subseteq E(\mathbb{G}_{\sigma})$ correspond bijectively to the edge ribbons of $E_{PD}^{\sigma} \subseteq E(\mathbb{G}_{\sigma_{bl}})$. The remaining results follow immediately. 
\end{proof}

In the following corollary, we translate Theorem~\ref{introthm} from the language of the Tait graph $G$ and the edge subset $E_{\sigma} \subseteq E(G)$ to the language of the black checkerboard state ribbon graph $\mathbb{G}_{\sigma_{bl}}$ and the edge subset $E_{PD}^{\sigma} \subseteq E(\mathbb{G}_{\sigma_{bl}})$. This result follows immediately from Proposition~\ref{Gprop} and Proposition~\ref{edgeprop}.

\begin{corollary} \label{bigcoralt}
Let $D$ be a connected, reduced, checkerboard-colored link diagram with at least one crossing and let $\mathbb{G}_{\sigma_{bl}}$ denote the black checkerboard state ribbon graph of $D$. Then $D$ is $\sigma$-adequate with respect to a state $\sigma$ if and only if there exists a partition $E(\mathbb{G}_{\sigma_{bl}})=E_{PD}^{\sigma} \sqcup E_{F}^{\sigma}$ of the edge ribbons of $\mathbb{G}_{\sigma_{bl}}$ such that the following conditions hold. 
\begin{enumerate}
	\item[(1)] Every edge ribbon of $E_{PD}^{\sigma}$ is contained in a cycle of $\mathbb{G}_{\sigma_{bl}}|E_{PD}^{\sigma}$. 
  \item[(2)] No edge ribbon of $E_{F}^{\sigma}$ has both endpoints on a connected component of $\mathbb{G}_{\sigma_{bl}}|E_{PD}^{\sigma}$. 
\end{enumerate}
\end{corollary}

In the remark below, we study how partial duality affects the black checkerboard state ribbon graph $\mathbb{G}_{\sigma_{bl}}$.

\begin{remark} \label{PDstagesrmk}
Let $D$ be a connected, reduced, checkerboard-colored link diagram with at least one crossing, let $\sigma$ be a $\sigma$-adequate state of $D$ with associated state ribbon graph $\mathbb{G}_{\sigma}$, and let $\sigma_{bl}$ denote the black checkerboard state of $D$ with associated black checkerboard state ribbon graph $\mathbb{G}_{\sigma_{bl}}$. Then we can realize $\mathbb{G}_{\sigma}$ as a partial dual of $\mathbb{G}_{\sigma_{bl}}$ in three steps.

\bigskip

\noindent \underline{Step 1}: First, delete the edge ribbons of $E_{F}^{\sigma}$ from $\mathbb{G}_{\sigma_{bl}}$ to obtain the graph $\mathbb{G}_{\sigma_{bl}}|{E_{PD}^{\sigma}}$. Since Corollary~\ref{bigcoralt} requires every edge ribbon of $E_{PD}^{\sigma}$ to be contained in a cycle of $\mathbb{G}_{\sigma_{bl}}|E_{PD}^{\sigma}$ for $D$ to be $\sigma$-adequate with respect to the state $\sigma$, then $\mathbb{G}_{\sigma_{bl}}|{E_{PD}^{\sigma}}$ consists of a union of isolated vertex disks and cycles of $\mathbb{G}_{\sigma_{bl}}$. 

\bigskip

\noindent \underline{Step 2}: Second, we apply the operation of partial duality to dualize the (non-loop) edge ribbons of $E_{PD}^{\sigma}$ and obtain the graph $\mathbb{G}_{\sigma}|{(E_{PD}^{\sigma})^{*}}$. 

\bigskip

\noindent \underline{Step 3}: Finally, we add the edge ribbons of $E_{F}^{\sigma}$ back to their original locations, relabeling this collection of edge ribbons as $(E_{F}^{\sigma})^{*}$ and noting that the structure of the vertex disks may have been changed by the operation of partial duality. 

\bigskip

\noindent See Figure~\ref{PDExample2} for a depiction of the construction of the state ribbon graph $\mathbb{G}_{\sigma}$, as realized as a partial dual of the black checkerboard state ribbon graph $\mathbb{G}_{\sigma_{bl}}$. Note that the corresponding connected, reduced, checkerboard-colored link diagram $D$ with $\sigma$-adequate state $\sigma$ has been suppressed. Upon closer investigation, the following conditions can be shown to occur during the three-step partial duality process described above. Figure~\ref{PDExample1} and Figure~\ref{PDExample2} provide useful examples.  
\begin{enumerate}
	\item[(1)] Each isolated vertex disk of $\mathbb{G}_{\sigma_{bl}}|{E_{PD}^{\sigma}}$ becomes an isolated vertex disk of $\mathbb{G}_{\sigma}|{(E_{PD}^{\sigma})^{*}}$, which becomes a non-isolated vertex disk of $\mathbb{G}_{\sigma}$.  
	\item[(2)] The inner boundary of each fundamental cycle in a connected component of $\mathbb{G}_{\sigma_{bl}}|{E_{PD}^{\sigma}}$ becomes an \textit{inner vertex disk} of $\mathbb{G}_{\sigma}|{(E_{PD}^{\sigma})^{*}}$, which becomes a vertex disk of $\mathbb{G}_{\sigma}$. 
  \item[(3)] The outer boundary of each connected component of $\mathbb{G}_{\sigma_{bl}}|{E_{PD}^{\sigma}}$ becomes an \textit{outer vertex disk} of $\mathbb{G}_{\sigma}|{(E_{PD}^{\sigma})^{*}}$, which becomes a vertex disk of $\mathbb{G}_{\sigma}$.
	\item[(4)] The edge ribbons of a fundamental cycle of $\mathbb{G}_{\sigma_{bl}}|{E_{PD}^{\sigma}}$ become edge ribbons of $\mathbb{G}_{\sigma}|{(E_{PD}^{\sigma})^{*}}$ that emanate from an inner vertex disk.
	\item[(5)] If two fundamental cycles of $\mathbb{G}_{\sigma_{bl}}|{E_{PD}^{\sigma}}$ share an edge ribbon, then the resulting ribbon subgraphs of $\mathbb{G}_{\sigma}|{(E_{PD}^{\sigma})^{*}}$ share a corresponding dual edge ribbon.   
	\item[(6)] The fixed edge ribbons of $E_{F}^{\sigma}$ become the fixed edge ribbons of $(E_{F}^{\sigma})^{*}$ where
	\begin{itemize}
		\item the edge ribbons of $E_{F}^{\sigma}$ that are inside a fundamental cycle of $\mathbb{G}_{\sigma_{bl}}|{E_{PD}^{\sigma}}$ and are incident to vertex disks of a fundamental cycle of $\mathbb{G}_{\sigma_{bl}}|{E_{PD}^{\sigma}}$ become edge ribbons of $(E_{F}^{\sigma})^{*}$ that are incident to inner vertex disks of $\mathbb{G}_{\sigma}|{(E_{PD}^{\sigma})^{*}}$. 
		\item the edge ribbons of $E_{F}^{\sigma}$ that are in the unbounded region $\mathbb{G}_{\sigma_{bl}}|{E_{PD}^{\sigma}}$ and are incident to vertex disks of $\mathbb{G}_{\sigma_{bl}}|{E_{PD}^{\sigma}}$ become edge ribbons of $(E_{F}^{\sigma})^{*}$ that are incident to outer vertex disks of $\mathbb{G}_{\sigma}|{(E_{PD}^{\sigma})^{*}}$. 
	\end{itemize}
\end{enumerate}
\end{remark}

\begin{figure} 
\centering
\def\svgwidth{6in}
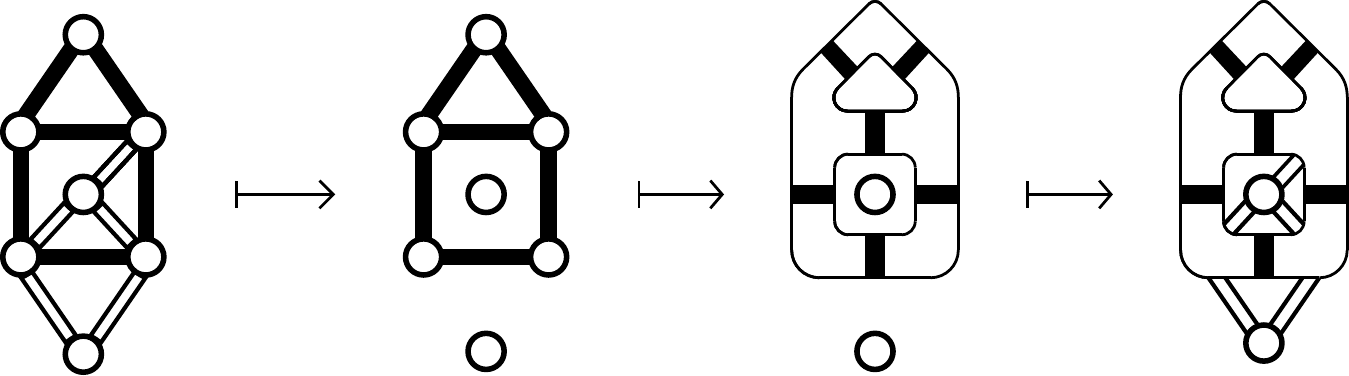
\caption{A depiction of the construction of the state ribbon graph $\mathbb{G}_{\sigma}$, as realized as a partial dual of the black checkerboard state ribbon graph $\mathbb{G}_{\sigma_{bl}}$. The edge ribbons of $E_{PD}^{\sigma}$ and $(E_{PD}^{\sigma})^{*}$ are shaded black while the edge ribbons of $E_{F}^{\sigma}$ and $(E_{F}^{\sigma})^{*}$, as well as all of the vertex disks, are shaded white. Recall that we view disks coming from inner state circles as lying above disks coming from outer state circles in order to avoid any self intersections in the resulting ribbon graph. The ribbon graphs, in order from left to right, are $\mathbb{G}_{\sigma_{bl}}$, $\mathbb{G}_{\sigma_{bl}}|{E_{PD}^{\sigma}}$, $\mathbb{G}_{\sigma}|{(E_{PD}^{\sigma})^{*}}$, and $\mathbb{G}_{\sigma}$.}
\label{PDExample2}
\end{figure}

\section{Sigma-Homogeneous and Homogeneously Adequate States of a Link Diagram} \label{homogmethod}

In this section, we define the families of \textit{$\sigma$-homogeneous} and \textit{homogeneously adequate} link diagrams, we prove Theorem~\ref{maincor2}, and we provide a method for finding all of the homogeneously adequate states of a link diagram.

\subsection{Sigma-Homogeneous States of a Link Diagram}

To begin, we define what it means for a link diagram to be \textit{$\sigma$-homogeneous} with respect to a state $\sigma$. 

\begin{definition} \label{homogeneous}
Let $D$ be a link diagram and let $\sigma$ be a state of $D$. Given that the collection of state circles of the state graph $H_{\sigma}$, which we will denote by $s_{\sigma}(D)$, divide the plane $\mathbb{R}^{2}$ into \textit{complementary regions}, we say that $D$ is \textit{$\sigma$-homogeneous with respect to a state $\sigma$} if no complementary region of $s_{\sigma}(D)$ contains both A-segments and B-segments. We call such a state a \textit{$\sigma$-homogeneous state} of $D$. Note that the all-A and all-B states of $D$ are always $\sigma$-homogeneous states. 
\end{definition}

\begin{remark}
It can be shown, by using Figure~\ref{res}, that every state $\sigma$ of an alternating link diagram $D$ is a $\sigma$-homogeneous state. It can also be shown that, given any link diagram $D$, all of the state segments of the black (resp. white) checkerboard state graph $H_{\sigma_{bl}}$ (resp. $H_{\sigma_{wh}}$) lie in a single complementary region of $s_{\sigma_{bl}}(D)$ (resp. $s_{\sigma_{wh}}(D)$). Therefore, by Proposition~\ref{checkerrmk}, we can show that $D$ is alternating if and only if both of the checkerboard states $\sigma_{bl}$ and $\sigma_{wh}$ are $\sigma$-homogeneous states of $D$.   
\end{remark}

\begin{remark} \label{complregremark}
Let $D$ be a connected, reduced, checkerboard-colored link diagram with at least one crossing, let $\sigma$ be a $\sigma$-adequate state of $D$ with associated state ribbon graph $\mathbb{G}_{\sigma}$, and let $\sigma_{bl}$ denote the black checkerboard state of $D$ with associated black checkerboard state ribbon graph $\mathbb{G}_{\sigma_{bl}}$. Given Remark~\ref{PDstagesrmk} (and using Figure~\ref{PDExample2} as a motivating example), we can classify the nonempty complementary regions of $s_{\sigma}(D)$ as follows.  
\begin{enumerate}
	\item[(1)] The ribbon subgraph of $\mathbb{G}_{\sigma}$ coming from a connected component of $\mathbb{G}_{\sigma_{bl}}|{E_{PD}^{\sigma}}$ corresponds to a single nonempty complementary region of $s_{\sigma}(D)$, namely the single nonempty complementary region between the boundary of an outer vertex disk and a collection of boundaries of inner vertex disks.  
	\item[(2)] The ribbon subgraph of $\mathbb{G}_{\sigma}$ coming from the portion of $\mathbb{G}_{\sigma_{bl}}|E_{F}^{\sigma}$ contained inside a fundamental cycle of $\mathbb{G}_{\sigma_{bl}}|{E_{PD}^{\sigma}}$ corresponds to a single nonempty complementary region of $s_{\sigma}(D)$, namely the single nonempty bounded complementary region inside the boundary of an inner vertex disk. 
  \item[(3)] The ribbon subgraph of $\mathbb{G}_{\sigma}$ coming from the portion of $\mathbb{G}_{\sigma_{bl}}|E_{F}^{\sigma}$ contained in the unbounded region of $\mathbb{G}_{\sigma_{bl}}|{E_{PD}^{\sigma}}$ corresponds to a single nonempty complementary region of $s_{\sigma}(D)$, namely the single nonempty unbounded complementary region outside all of boundaries of the outer vertex disks.
\end{enumerate}  
\end{remark}

\subsection{Homogeneously Adequate States of a Link Diagram}

We now move on to study link diagrams that are both $\sigma$-adequate and $\sigma$-homogeneous with respect to a state $\sigma$. These link diagrams, called \textit{homogeneously adequate} link diagrams, have been explored by a number of authors (\cite{Homogeneous}, \cite{Guts}, \cite{Survey}, \cite{Ozawa}). 

\begin{definition} \label{adeqhomogeneous}
Let $D$ be a link diagram and let $\sigma$ be a state of $D$. If $D$ is $\sigma$-adequate and $\sigma$-homogeneous with respect to the state $\sigma$, then we call $D$ \textit{homogeneously adequate with respect to the state $\sigma$}. Note that an A-adequate (resp. B-adequate) link diagram is homogeneously adequate with respect to the all-A (resp. all-B) state $\sigma_{A}$ (resp. $\sigma_{B}$). 
\end{definition}

We now give a characterization of when a link diagram is homogeneously adequate with respect to a state $\sigma$. 

\begin{proposition} \label{bigthm}
Let $D$ be a connected, reduced, checkerboard-colored link diagram with at least one crossing and let $\mathbb{G}_{\sigma_{bl}}$ denote the black checkerboard state ribbon graph of $D$. Then $D$ is homogeneously adequate with respect to a state $\sigma$ if and only if there exists a partition $E(\mathbb{G}_{\sigma_{bl}})=E_{PD}^{\sigma} \sqcup E_{F}^{\sigma}$ of the edge ribbons of $\mathbb{G}_{\sigma_{bl}}$ such that the following conditions hold. 
\begin{enumerate}
	\item[(1)] Every edge ribbon of $E_{PD}^{\sigma}$ is contained in a cycle of $\mathbb{G}_{\sigma_{bl}}|E_{PD}^{\sigma}$. 
  \item[(2)] No edge ribbon of $E_{F}^{\sigma}$ has both endpoints on a connected component of $\mathbb{G}_{\sigma_{bl}}|E_{PD}^{\sigma}$. 
	\item[(3)] Each connected component of $\mathbb{G}_{\sigma_{bl}}|E_{PD}^{\sigma}$ has edge ribbons corresponding to crossings of $D$ that are either all A-resolved or all B-resolved according to the state $\sigma$. 
	\item[(4)] The edge ribbons of $E_{F}^{\sigma}$ inside a fundamental cycle of $\mathbb{G}_{\sigma_{bl}}|E_{PD}^{\sigma}$ correspond to crossings of $D$ that are either all A-resolved or all B-resolved according to the state $\sigma$.  
 	\item[(5)] The edge ribbons of $E_{F}^{\sigma}$ in the unbounded region outside all of the fundamental cycles of $\mathbb{G}_{\sigma_{bl}}|E_{PD}^{\sigma}$ correspond to crossings of $D$ that are either all A-resolved or all B-resolved according to the state $\sigma$. 
\end{enumerate}
\end{proposition}

\begin{proof}
$(\Leftarrow)$ Conclusion~(1) and Conclusion~(2) imply, by Corollary~\ref{bigcoralt}, that $D$ is $\sigma$-adequate with respect to the state $\sigma$. Conclusion~(3), Conclusion~(4), and Conclusion~(5) imply, by Definition~\ref{homogeneous} and Remark~\ref{complregremark}, that $D$ is $\sigma$-homogeneous with respect to the state $\sigma$.  

$(\Rightarrow)$ Since $D$ is $\sigma$-adequate with respect to the state $\sigma$, then Corollary~\ref{bigcoralt} implies that Conclusion~(1) and Conclusion~(2) must hold. Let $H_{\sigma}$ denote the state graph and let $\mathbb{G}_{\sigma}$ denote the state ribbon graph associated to the state $\sigma$ of $D$. By Proposition~\ref{dualcor}, $\mathbb{G}_{\sigma}$ can be realized as a partial dual of the black checkerboard state ribbon graph $\mathbb{G}_{\sigma_{bl}}$, which we know to be loopless by Remark~\ref{nolooprmk}. Therefore, we can apply partial duality to $\mathbb{G}_{\sigma_{bl}}$ using the three steps given by Remark~\ref{PDstagesrmk} and can, by Remark~\ref{complregremark}, classify the nonempty complementary regions of $s_{\sigma}(D)$ into three types. Since $D$ is $\sigma$-homogeneous with respect to the state $\sigma$, then, by Definition~\ref{homogeneous}, each nonempty complementary region must either contain only A-segments or only B-segments. Thus, Conclusion~(3), Conclusion~(4), and Conclusion~(5) must hold.    
\end{proof}

We now rephrase Proposition~\ref{bigthm}, changing the focus from the A- and B-resolutions of crossings of $D$ to the $+$ and $-$ edge signs of the Tait graph $G$ that forms the spine of the black checkerboard state ribbon graph $\mathbb{G}_{\sigma_{bl}}$. 

\begin{proposition} \label{bigthm2}
Let $D$ be a connected, reduced, checkerboard-colored link diagram with at least one crossing and with associated Tait graph $G$. Let $\mathbb{G}_{\sigma_{bl}}$ denote the black checkerboard state ribbon graph of $D$. Then $D$ is homogeneously adequate with respect to a state $\sigma$ if and only if there exists a partition $E(\mathbb{G}_{\sigma_{bl}})=E_{PD}^{\sigma} \sqcup E_{F}^{\sigma}$ of the edge ribbons of $\mathbb{G}_{\sigma_{bl}}$ such that the following conditions hold. 
\begin{enumerate}
	\item[(1)] Every edge ribbon of $E_{PD}^{\sigma}$ is contained in a cycle of $\mathbb{G}_{\sigma_{bl}}|E_{PD}^{\sigma}$. 
  \item[(2)] No edge ribbon of $E_{F}^{\sigma}$ has both endpoints on a connected component of $\mathbb{G}_{\sigma_{bl}}|E_{PD}^{\sigma}$. 
	\item[(3)] Each connected component of $\mathbb{G}_{\sigma_{bl}}|E_{PD}^{\sigma}$ corresponds to a connected component of $G|E_{\sigma}$ and this connected component either consists entirely of $+$ edges or consists entirely of $-$ edges. 
	\item[(4)] The edge ribbons of $E_{F}^{\sigma}$ inside a fundamental cycle of $\mathbb{G}_{\sigma_{bl}}|E_{PD}^{\sigma}$ correspond to the edges of $\overline{E_{\sigma}}$ inside a fundamental cycle of $G|E_{\sigma}$ and these edges either consist entirely of $+$ edges or consist entirely of $-$ edges. 
 	\item[(5)] The edge ribbons of $E_{F}^{\sigma}$ in the unbounded region outside all of the fundamental cycles of $\mathbb{G}_{\sigma_{bl}}|E_{PD}^{\sigma}$ correspond to the edges of $\overline{E_{\sigma}}$ outside all of the fundamental cycles of $G|E_{\sigma}$ and these edges either consist entirely of $+$ edges or consist entirely of $-$ edges. 
\end{enumerate}
\end{proposition}

\begin{proof}
To prove this proposition, it suffices to show that Conclusion~(3), Conclusion~(4), and Conclusion~(5) of Proposition~\ref{bigthm} are equivalent to the corresponding conclusions of this proposition. By Proposition~\ref{signlemma} and Proposition~\ref{Gprop}, given the black checkerboard state $\sigma_{bl}$ of $D$, every edge of $G$ (and, therefore, every edge of $\mathbb{G}_{\sigma_{bl}}$) is either a $(+,B)$ edge or a $(-,A)$ edge. Therefore, we have that $E_{\sigma_{bl}} = \emptyset$ (and, therefore, that $E_{PD}^{\sigma_{bl}} = \emptyset$). By Proposition~\ref{edgeprop}, the edges in the subset $E_{\sigma} \subseteq E(G)$, which are either $(+,A)$ or $(-,B)$ edges, correspond bijectively to the edge ribbons in the subset $(E_{PD}^{\sigma})^{*} \subseteq E(\mathbb{G}_{\sigma})$ and the edges in the subset $\overline{E_{\sigma}} \subseteq E(G)$, which are either $(+,B)$ or $(-,A)$ edges, correspond bijectively to the edge ribbons in the subset $(E_{F}^{\sigma})^{*} \subseteq E(\mathbb{G}_{\sigma})$. 

This says that choosing an edge partition $E(\mathbb{G}_{\sigma_{bl}})=E_{PD}^{\sigma} \sqcup E_{F}^{\sigma}$ of the edge ribbons of $\mathbb{G}_{\sigma_{bl}}$ corresponds to choosing the edge ribbons of $\mathbb{G}_{\sigma_{bl}}$ that will become either $(+,A)$ or $(-,B)$ edge ribbons after partial duality turns $\mathbb{G}_{\sigma_{bl}}$ into $\mathbb{G}_{\sigma}$ and, consequently, choosing the edge ribbons of $\mathbb{G}_{\sigma_{bl}}$ that will stay fixed as either $(+,B)$ or $(-,A)$ edge ribbons after partial duality turns $\mathbb{G}_{\sigma_{bl}}$ into $\mathbb{G}_{\sigma}$. Since the edge ribbons of each connected component of $\mathbb{G}_{\sigma_{bl}}|E_{PD}^{\sigma}$ become either $(+,A)$ or $(-,B)$ edge ribbons, then statements about A- and B-resolved crossings of $D$ can be translated to statements about $+$ and $-$ edges of $G$ (and, therefore, $\mathbb{G}_{\sigma_{bl}}$). Similarly, since the edge ribbons of $E_{F}^{\sigma}$ stay fixed as either $(+,B)$ or $(-,A)$ edge ribbons, then statements about A- and B-resolved crossings of $D$ can be translated to statements about $+$ and $-$ edges of $G$ (and, therefore, $\mathbb{G}_{\sigma_{bl}}$).  
\end{proof}

We are now able to prove Theorem~\ref{maincor2} from the introduction. 

\begin{proof}[Proof of Theorem~\ref{maincor2}] 
By applying Proposition~\ref{Gprop} and Proposition~\ref{edgeprop}, we are able to translate Proposition~\ref{bigthm2} from the language of the black checkerboard state ribbon graph $\mathbb{G}_{\sigma_{bl}}$ and the edge subset $E_{PD}^{\sigma} \subseteq E(\mathbb{G}_{\sigma_{bl}})$ to the language of the Tait graph $G$ and the edge subset $E_{\sigma} \subseteq E(G)$. This gives the desired result. 
\end{proof}

Theorem~\ref{maincor2} provides a way to use the Tait graph $G$ associated to a checkerboard-colored link diagram $D$ to look for all homogeneously adequate states of $D$. 

\begin{remark}
In \cite{Ozawa}, Ozawa claims that the algebraic link diagram given in his Figure~7, call it $D$, has no homogeneously adequate states. This claim can be proved by using Theorem~\ref{maincor2} as follows. Without loss of generality, give $D$ the canonical checkerboard coloring and construct the associated Tait graph $G$. See Figure~\ref{Ozawa} for a depiction of $G$. For a contradiction, suppose $D$ has a homogeneously adequate state, call it $\sigma$. Then, by Theorem~\ref{maincor2}, Conclusions~(1) through (5) must hold. First, suppose $\sigma=\sigma_{bl}$ is the black checkerboard state. Then $E_{\sigma_{bl}} = \emptyset$ and, therefore, $G|E_{\sigma_{bl}}$ contains all of the vertices of $G$ but has no edges. This violates Conclusion~(5) of Theorem~\ref{maincor2}. See Figure~\ref{Ozawa}. Now suppose $\sigma \neq \sigma_{bl}$ is not the black checkerboard state. Then $E_{\sigma} \neq \emptyset$ and, therefore, $G|E_{\sigma_{bl}}$ contains edges. By Conclusions~(1) and (3) of Theorem~\ref{maincor2}, each connected component of $G|E_{\sigma}$ must be a union of cycles, all of whose edges are $+$ edges or all of whose edges are $-$ edges. This means that $E_{\sigma}$ cannot contain any of the edges of the path subgraph of length three with all $+$ edges since there is no cycle of $G$ with all $+$ edges that contains any of these edges. See Figure~\ref{Ozawa}. By a similar argument, $E_{\sigma}$ cannot contain any of the edges of the path subgraph of length three with all $-$ edges. Therefore, the edges in these two paths of length three are forced to be contained in the unbounded region outside all of the fundamental cycles of $G|E_{\sigma}$. This violates Conclusion~(5) of Theorem~\ref{maincor2}. 
\end{remark}

\begin{figure} 
\centering
\def\svgwidth{2in}
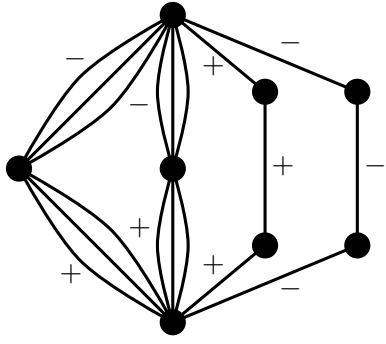
\caption{The Tait graph $G$ corresponding to the canonical checkerboard coloring of the algebraic link diagram given in Figure~7 of \cite{Ozawa}. All of the multiedges in the top half of the graph are $-$ edges and all of the multiedges in the bottom half of the graph are $+$ edges. The right side of the graph contains, as subgraphs, two path graphs of length three, one with all $+$ edges and one with all $-$ edges.}
\label{Ozawa}
\end{figure}

\subsection{A Method for Finding All Homogeneously Adequate States of a Link Diagram}

In Section~\ref{methodsec}, we presented a method for finding all of the $\sigma$-adequate states of a connected, reduced, checkerboard-colored link diagram $D$. In this section, we add to this method so that all homogeneously adequate states of such a link diagram can be found. 

\bigskip

\noindent \textbf{\underline{Steps 1 through 6}:} Use the method from Section~\ref{methodsec} to find all of the $\sigma$-adequate states of the given link diagram. 

\bigskip

\noindent \textbf{\underline{Step 7}:} Use Conclusion~(3), Conclusion~(4), and Conclusion~(5) of Theorem~\ref{maincor2} to determine which, if any, of the $\sigma$-adequate states found above are also $\sigma$-homogeneous with respect to the same state. 
 
\bigskip

\noindent \textbf{\underline{An Example of the Method}:} Return to the example (from Section~\ref{methodsec}) of the diagram $D$ of the knot $11n_{95}$. 

\bigskip

\noindent \underline{Steps 1 through 6}: In Section~\ref{methodsec}, we found the 20 $\sigma$-adequate states of $D$.  

\bigskip

\noindent \underline{Step 7}: By applying Theorem~\ref{maincor2} to Figure~\ref{11n95}, we get that none of the $\sigma$-adequate states of $D$ are also $\sigma$-homogeneous with respect to the given state. Specifically, by checking through all 20 $\sigma$-adequate states, only four states satisfy Conclusion~(3) of Theorem~\ref{maincor2}. These are the black checkerboard state $\sigma_{bl}$ and the three states $\sigma_{1}$, $\sigma_{2}$, and $\sigma_{6}$ with corresponding edge subsets $E_{\sigma}$ forming fundamental cycles of $G$ that bound no regions, Region~1, Region~3, and both Region~1 and Region~3 (respectively). See Table~\ref{statetable}. In each of these four cases, Conclusion~(4) of Theorem~\ref{maincor2} holds vacuously but Conclusion~(5) of Theorem~\ref{maincor2} fails. Therefore, the standard diagram of the knot $11n_{95}$ has 20 $\sigma$-adequate states but no homogeneously adequate states. 

Note that we can use Theorem~\ref{maincor2} to directly search for homogeneously adequate states of $D$. While this method may be quicker, the advantage to using the method of this section is that Theorem~\ref{mainthm1} and Theorem~\ref{mainthm2} are utilized as a means to confirm, using computations of the symmetrized Tutte polynomial $\chi_{G}(t,t)$ and the polynomials $\phi_{D}^{\sigma}(t)$, that all of the $\sigma$-adequate states of $D$ have been found before Theorem~\ref{maincor2} is utilized.

\bibliography{mybib}
\bibliographystyle{plain}

\end{document}